\documentclass[aop,MSNbibl,dvips]{arximspdf}
\usepackage{mathrsfs}

\doi{10.1214/12-AOP805} %kopijuoti is PTS
\volume{42}
\issue{3}
\pubyear{2014}
\firstpage{865}
\lastpage{905}

\makeatletter

\newcommand{\rrvert}{\vert}
\newcommand{\llvert}{\vert}
\newtheorem{theorem}{Theorem}
\newtheorem{lemma}{Lemma}
\newtheorem{proposition}{Proposition}

\newproclaim{example}{Example}
\newproclaim{definition}{Definition}

\newproclaim{remark}{Remark}
\newproclaim{assumption}{Assumption}
\newcommand{\cF}{\mathcal{F}}
\newcommand{\cO}{\mathcal{O}}

\newcommand{\bbN}{\mathbb{N}}
\newcommand{\bbR}{\mathbb{R}}

\newcommand{\N}{\mathbb{N}}

\newcommand{\Ne}{\N^{\ast}}
\newcommand{\R}{\mathbb{R}}
\newcommand{\p}{^{\prime}}

\newcommand{\eqref}[1]{(\ref{#1})}
\makeatother

\begin{document}
\begin{frontmatter}

\title{The obstacle problem for quasilinear stochastic PDEs: Analytical approach}
\runtitle{The Obstacle problem for quasilnear SPDEs}

\begin{aug}
\author[A]{\fnms{Laurent}~\snm{Denis}\ead[label=e1]{ldenis@univ-evry.fr}\thanksref{T1}},
\author[B]{\fnms{Anis} \snm{Matoussi}\corref{}\ead[label=e2]{anis.matoussi@univ-lemans.fr}\thanksref{T2}}
\and
\author[A]{\fnms{Jing} \snm{Zhang}\ead[label=e3]{jing.zhang.etu@gmail.com}}
\thankstext{T1}{Supported in part by the chair \textit{risque de cr\'
edit}, F\'ed\'eration Bancaire Fran\c{c}aise.}
\thankstext{T2}{Supported in part by the Chair \textit{Financial Risks}
of the \textit{Risk Foundation} sponsored
by Soci\'et\'e G\'en\'erale, the Chair \textit{Derivatives of the Future}
sponsored by the {F\'ed\'eration Bancaire Fran\c{c}aise} and the Chair
\textit{Finance and Sustainable Development}
sponsored by EDF and Calyon.}
\runauthor{L. Denis, A. Matoussi and J. Zhang}
\affiliation{Universit{\'e} d'Evry Val d'Essonne,
Universit\'{e} du Maine and CMAP,
and Universit{\'e} d'Evry Val d'Essonne}
\address[A]{L. Denis \\
J. Zhang\\
Laboratoire Analyse et Probabilit\'es\\
Universit\'{e} d'Evry-Val-d'Essonne\\
23 Boulevard de France\\
F-91037 EVRY Cedex\\
France\\
%Universit{\'e} d'Evry Val d'Essonne\\
%Rue du P\`ere Jarlan\\
%F-91025 Evry Cedex\\
%France\\
\printead{e1}\\
\phantom{E-mail:\ }\printead*{e3}}
\address[B]{A. Matoussi\\
Laboratoire Manceau de Math\'{e}matiques\\
Universit\'{e} du Maine\\
F\'{e}d\'{e}ration de Recherche 2962 du CNRS\\
Math\'{e}matiques des Pays de Loire\\
Avenue Olivier Messiaen\\
F-72085 Le Mans Cedex 9\\
France\\
and\\
CMAP\\
Ecole Polytechnique, Palaiseau\\
France\\
\printead{e2}} %adresu isvedimo komanda gale!
\end{aug}

% HISTORY:
\received{\smonth{2} \syear{2012}}
\revised{\smonth{9} \syear{2012}}

% ABSTRACT
%
\begin{abstract}
We prove an existence and uniqueness result
for quasilinear Stochastic PDEs with obstacle (OSPDE in short).
Our method is based on analytical technics coming from the parabolic
potential theory. The solution
is expressed as a pair $(u,\nu)$ where $u$ is a predictable
continuous process which takes values in a proper Sobolev space and
$\nu$ is a random regular measure satisfying the minimal Skohorod
condition.
\end{abstract}

% KEYWORDS
\pdfkeywords{60H15, 35R60, 31B150, Parabolic potential, regular measure, stochastic partial
differential equations, obstacle problem, penalization method, Ito's formula, comparison theorem,
space-time white noise}
\end{frontmatter}

%s1 #&#
\section{Introduction}
The starting point of this work is the following parabolic stochastic
partial differential equation (in short, SPDE):
%
%e1 #&#
\begin{eqnarray}
\label{SPDE}du_t(x)&=&\partial_i \bigl(a_{i,j}(x)
\partial_ju_t(x)+g_i\bigl(t,x,u_t(x),
\nabla u_t(x)\bigr) \bigr)\,dt\nonumber\\
&&{}+f\bigl(t,x,u_t(x),\nabla
u_t(x)\bigr)\,dt
\\
&&{}+\sum_{j=1}^{+\infty}h_j
\bigl(t,x,u_t(x),\nabla u_t(x)\bigr)\,dB^j_t,\nonumber
\end{eqnarray}
where $a$ is a symmetric bounded measurable matrix which defines a
second order
operator on an open domain
$\mathcal{O}\subset\bbR^d$, with Dirichlet boundary condition. The initial
condition is given as $u_0=\xi$, a $L^2(\mathcal{O})$-valued random
variable, and $f$, $g=(g_1,\ldots,g_d)$ and $h=(h_1,\ldots, h_i,\ldots )$ are
nonlinear random functions. Given an obstacle $S\dvtx \Omega\times
[0,T]\times\cO\rightarrow\R$, we study the obstacle problem for SPDE~(\ref{SPDE}),
that is, we want to find a solution of (\ref{SPDE}) which
satisfies ``$u\geq S$'' where the obstacle~$S$ is regular in some sense and
controlled by the solution of a SPDE.

Nualart and Pardoux \cite{NualartPardoux} have studied the obstacle
problem for a nonlinear heat equation on the spatial interval
$[0,1]$ with Dirichlet boundary conditions, driven by an additive
space--time white noise. They proved the existence and uniqueness of
the solution and their method relied heavily on the results for a
deterministic variational inequality. Donati-Martin and Pardoux
\cite{DonatiPardoux} generalized the model of Nualart and Pardoux.
The nonlinearity appears both in the drift and in the diffusion coefficients.
They proved the existence of the solution by penalization method but
they did not obtain the uniqueness result. And then in 2009, Xu and
Zhang solved the problem of the uniqueness; see \cite{XuZhang}.
However, in all their models,
there is not the term of divergence and they do not consider the case
where the coefficients depend on $\nabla u$.

The work of El Karoui et al. \cite{EKPPQ} treats the obstacle problem
for deterministic semilinear PDE's within the framework of
backward stochastic differential equations (BSDE in short). Namely,
the equation \eqref{SPDE} is considered with $f$ depending on~$u$
and $\nabla u$, while the function $g$ is null (as well~$h$) and the
obstacle $v$ is continuous. They considered the viscosity solution
of the obstacle problem for the equation \eqref{SPDE}, they
represented this solution stochastically as a process and the main
new object of this BSDE framework is a continuous increasing
process that controls the set $\{u=v\}$. Bally et al. \cite{BCEF}
(see also \cite{MX08}) point out that the continuity of this
process allows one to extend the classical notion of a strong
variational solution (see Theorem 2.2 of \cite{BensoussanLions78},
page 238)
and express the solution to the obstacle as a pair $(u, \nu)$ where $
\nu$ is supported by the set $\{u=v\}$.

Matoussi and Stoica \cite{MatoussiStoica} have proved an existence and
uniqueness result for the obstacle problem of backward quasilinear
stochastic PDE on the whole space $\bbR^d$ and driven by a finite
dimensional Brownian motion. The method is based on the probabilistic
interpretation of the solution by using the backward doubly stochastic
differential equation (BDSDE in short). They have also proved that
the solution is a pair $ ( u,\nu )$ where $u$ is a predictable
continuous process which takes values in a proper Sobolev space and
$\nu$ is a random regular measure satisfying the minimal Skohorod
condition. In particular, they gave for the regular measure $\nu$ a
probabilistic interpretation in terms of the continuous increasing
process $K$ where $ (Y,Z,K)$ is the solution of a reflected generalized BDSDE.

Michel Pierre \cite{Pierre,PIERRER} has studied the parabolic PDE with obstacle
using the parabolic potential as a tool. He proved that the solution
uniquely exists and is quasi-continuous. With the help of Pierre's
result, under suitable assumptions on $f$, $g$ and $h$, our aim is to
prove existence and uniqueness for the following SPDE with given
obstacle $S$ that we write formally as
%
%e2 #&#
%e3 #&#
%e4 #&#
%e5 #&#
%e6 #&#
\begin{eqnarray}
\label{SPDEO}\cases{\displaystyle du_t(x)=\partial_i \bigl(a_{i,j}(x)
\partial_ju_t(x)+g_i\bigl(t,x,u_t(x),
\nabla u_t(x)\bigr) \bigr)\,dt\vspace*{2pt}\cr
\hspace*{19pt}\qquad{}\displaystyle +f\bigl(t,x,u_t(x),\nabla
u_t(x)\bigr)\,dt\vspace*{2pt}\cr
\hspace*{19pt}\qquad{}
 \displaystyle+\sum_{j=1}^{+\infty
}h_j
\bigl(t,x,u_t(x),\nabla u_t(x)\bigr)\,dB^j_t,
\vspace*{2pt}\cr
\displaystyle u_t(x)\geq S_t(x),\qquad \forall(t,x)\in\bbR^+\times\cO,
\vspace*{2pt}\cr
 \displaystyle u_0(x)=\xi(x),\hspace*{3pt}\qquad\forall x\in\cO,
\vspace*{2pt}\cr
 \displaystyle u_t(x)=0,\hspace*{19pt}\qquad\forall(t,x)\in \bbR^+\times\partial\cO.}
\end{eqnarray}

To give a rigorous definition to the notion of a solution to this
equation, we will use the technics of parabolic potential theory
developed by M. Pierre
in the stochastic framework. We first prove a quasi-continuity result
for the solution of the SPDE \eqref{SPDE} with null Dirichlet condition
on given domain $\mathcal O$ and driven by an infinite dimensional
Brownian motion. This result is not obvious and is based
on a mixing pathwise argument and Mignot and Puel \cite{MignotPuel}
existence result
of the obstacle problem for some deterministic PDEs. Moreover, we prove
in our context that the reflected measure $\nu$ is a regular random
measure and we give
the analytical representation of such a measure in terms of parabolic
potential in the sense given by M. Pierre in \cite{Pierre}. The main
theorem we obtain is the following:

%th1 #&#
\begin{theorem}
Assume that $f$,
$g$ and $h$ satisfy some Lipschitz continuity and integrability
hypotheses, $\xi\in L^2 (\Omega\times\cO)$, $S$ is quasi-continuous and
$S_t\leq S'_t $, where~$S'$ is the solution of the linear SPDE with null boundary condition
\[
\cases{ %
\displaystyle dS'_t=LS'_t\,dt+f'_t\,dt+
\sum_{i=1}^d \partial_i
g'_{i,t}\,dt+\sum_{j=1}^{+\infty}h'_{j,t}\,dB^j_t,
\vspace*{2pt}\cr
S'(0)=S'_0,}
\]
where $S'_0\in L^2 (\Omega\times\cO)$, and $f'$, $g'$ and $h'$ are
square integrable adapted
processes.

Then there exists a unique solution $ ( u,\nu )$ of the
obstacle problem for the SPDE~(\ref{SPDEO}) associated to $(\xi,f,g,h,S)$, that is, $u$ is a predictable continuous process which
takes values in a proper Sobolev space, $u\geq S$ and
$\nu$ is a random regular measure such that:
\begin{longlist}[(1)]
\item[(1)] The following relation holds almost surely, for all
$t\in[0,T]$ and $\forall\varphi\in\mathcal{C}_{c}^{\infty} (\R^+ )\otimes\mathcal{C}_c^2 (\cO)$,
\begin{eqnarray*}
&&(u_t,\varphi_t)-(\xi,\varphi_0)-\int
_0^t(u_s,\partial_s
\varphi_s)\,ds+\int_0^t
\mathcal{E}(u_s,\varphi_s)\,ds\\
&&\quad{}+\sum
_{i=1}^d\int_0^t
\bigl(g^i_s(u_s,\nabla u_s),
\partial_i\varphi_s\bigr)\,ds
\\
&&\qquad=\int_0^t\bigl(f_s(u_s,
\nabla u_s),\varphi_s\bigr)\,ds+\sum
_{j=1}^{+\infty
}\int_0^t
\bigl(h^j_s(u_s,\nabla u_s),
\varphi_s\bigr)\,dB^j_s\\
&&\qquad\quad{}+\int
_0^t\int_{\mathcal
{O}}
\varphi_s(x)\nu(dx,ds).
\end{eqnarray*}
\item[(2)]$u$ admits a quasi-continuous version, $\tilde{u}$, and we have
the mininal Skohorod condition
\[
\int_0^T\int_{\mathcal{O}}\bigl(
\tilde{u}(s,x)-{S}(s,x)\bigr)\nu(dx,ds)=0\qquad \mbox{a.s.}
\]
\end{longlist}
\end{theorem}

This paper is divided as follows: in the second section we set the
assumptions, then we introduce in the third section the notion of
a regular measure associated to parabolic potentials. The fourth
section is devoted to prove the
quasi-continuity of the solution of SPDE without obstacle. The fifth
section is the main part of the paper in which we prove existence
and uniqueness of the solution. To do that, we begin with the linear
case, and
then by Picard iteration we get the result in the
nonlinear case; we also establish an It\^o formula. Finally, in the
sixth section we prove
a comparison theorem for the solution of SPDE with obstacle.

%s2 #&#
\section{Preliminaries}
We consider a sequence $((B^i(t))_{t\geq0})_{i\in\mathbb{N}^*}$ of
independent Brownian motions defined on a standard filtered
probability space $(\Omega,\mathcal{F},(\mathcal{F}_t)_{t\geq0}, P)$
satisfying the usual conditions.

Let $\mathcal{O}\subset\bbR^d$ be an open domain and
$L^2(\mathcal{O})$ the set of square integrable functions with
respect to the Lebesgue measure on $\mathcal{O}$. It is a Hilbert
space equipped with the usual scalar product and norm as follows:
\[
(u,v)=\int_\mathcal{O}u(x)v(x)\,dx, \qquad\Vert u\Vert =\biggl(\int
_\mathcal{O}u^2(x)\,dx\biggr)^{1/2}.
\]
Let $A$ be a symmetric second order differential operator, with domain
$\mathcal{D}(A)$, given by
\[
A:=-L=-\sum_{i,j=1}^d\partial_i
\bigl(a^{i,j}(x)\partial_j\bigr).
\]
We assume that
$a(x)=(a^{i,j}(x))_{i,j}$ is a measurable symmetric matrix defined on~$\mathcal{O}$ which satisfies the uniform ellipticity
condition
\[
\lambda|\xi|^2\leq\sum_{i,j=1}^d
a^{i,j}(x)\xi^i\xi^j\leq\Lambda|
\xi|^2\qquad \forall x\in\mathcal {O}, \xi\in\bbR^d,
\]
where $\lambda$ and $\Lambda$ are positive constants.

Let $(F,\mathcal{E})$ be the associated Dirichlet form given by
$F:=\mathcal{D}(A^{1/2})=H_0^1(\mathcal{O})$ and
\[
\mathcal{E}(u,v):=\bigl(A^{1/2}u,A^{1/2}v\bigr)\quad \mbox{and}\quad
\mathcal{E}(u)=\bigl\Vert A^{1/2}u\bigr\Vert^2\qquad \forall u,v\in F,
\]
where $H_0^1(\mathcal{O})$ is the first order
Sobolev space of functions vanishing at the boundary. As usual, we
shall denote $H^{-1}(\mathcal{O})$ its dual space.

We consider the quasilinear stochastic partial differential equation
(\ref{SPDE}) with initial condition $u(0,\cdot)=\xi(\cdot)$ and Dirichlet
boundary condition $u(t,x)=0, \forall (t,x)\in
\bbR^+\times\partial\mathcal{O}$.

We assume that we have predictable random
functions
\begin{eqnarray*}
 f\dvtx \bbR^+\times\Omega\times\mathcal{O}\times \bbR\times\bbR^d
&\rightarrow& \bbR,
\\
g=(g_1,\ldots,g_d)\dvtx \bbR^+\times\Omega\times
\mathcal {O}\times\bbR \times \bbR^d&\rightarrow &\bbR^d,
\\
h=(h_1,\ldots,h_i,\ldots )\dvtx \bbR^+\times\Omega
\times\mathcal {O}\times \bbR\times\bbR^d&\rightarrow&
\bbR^{\mathbb{N}^*}.
\end{eqnarray*}
In the sequel,
$|\cdot|$ will always denote the underlying Euclidean or $l^2$-norm.
For
example,
\[
\bigl|h(t,\omega,x,y,z)\bigr|^2=\sum_{i=1}^{+\infty}\bigl|h_i(t,
\omega,x,y,z)\bigr|^2.
\]

\renewcommand{\theassumption}{(H)}
\begin{assumption}\label{assH}
There exist nonnegative constants $C, \alpha, \beta$ such that for almost all $\omega$, the following
inequalities hold for all
$(t,x,y,z)\in\bbR^+\times\mathcal{O}\times\mathbb{R}\times
\mathbb{R}^d$:
\begin{longlist}[(1)]
\item[(1)]$|f(t,\omega,x,y,z)-f(t,\omega,x,y',z')|\leq C(|y-y'|+|z-z'|),$\vspace*{1pt}
\item[(2)]$(\sum_{i=1}^d|g_i(t,\omega,x,y,z)-g_i(t,\omega,x,y',z')|^2)^{{1}/{2}}\leq
C|y-y'|+\alpha|z-z'|,$\vspace*{1pt}
\item[(3)]$(|h(t,\omega,x,y,z)-h(t,\omega,x,y',z')|^2)^{{1}/{2}}\leq
C|y-y'|+\beta|z-z'|,$\vspace*{1pt}
\item[(4)] the contraction property: $2\alpha+\beta^2<2\lambda.$
\end{longlist}
\end{assumption}
%
%re1 #&#
\begin{remark} This last contraction property ensures existence and
uniqueness for the solution of the SPDE without obstacle (see \cite
{DenisStoica}).
\end{remark}

With the uniform ellipticity condition we have the following
equivalent conditions:
\begin{eqnarray*}
\bigl\Vert f(u,\nabla u)-f(v,\nabla v)\bigr\Vert &\leq& C\Vert u-v\Vert +C
\lambda^{-1/2}\mathcal{E}^{1/2}(u-v),
\\
\bigl\Vert g(u,\nabla u)-g(v,\nabla v)\bigr\Vert_{L^2(\mathcal{O};\bbR^d)}&\leq& C\Vert u-v\Vert +
\alpha\lambda^{-1/2}\mathcal{E}^{1/2}(u-v),
\\
\bigl\Vert h(u,\nabla u)-h(v,\nabla v)\bigr\Vert_{L^2(\mathcal{O};\bbR^{\mathbb{N}^*})}&\leq& C\Vert u-v\Vert +\beta
\lambda^{-1/2}\mathcal{E}^{1/2}(u-v).
\end{eqnarray*}

Moreover, for simplicity, we fix a terminal time $T >0$, and we assume
the following:

\renewcommand{\theassumption}{(I)}
\begin{assumption}\label{assI}
\begin{eqnarray*}
\xi&\in& L^2(\Omega\times\mathcal{O})\mbox{ is an }\mathcal{F}_0\mbox{-measurable
random variable},
\\
f(\cdot,\cdot,\cdot,0,0)&:=&f^0\in L^2\bigl([0,T]\times
\Omega\times\mathcal{O};\bbR\bigr),
\\
g(\cdot,\cdot,\cdot,0,0)&:=&g^0=\bigl(g_1^0,
\ldots,g_d^0\bigr)\in L^2\bigl([0,T]\times
\Omega\times\mathcal{O};\bbR^d\bigr),
\\
h(\cdot,\cdot,\cdot,0,0)&:=&h^0=\bigl(h_1^0,
\ldots,h_i^0,\ldots \bigr)\in L^2\bigl([0,T]
\times\Omega\times\mathcal{O};\bbR^{\mathbb{N}^*}\bigr).
\end{eqnarray*}
\end{assumption}

Now we introduce the notion of a weak solution.

We denote by
$\mathcal
{H}_T$ the space
of $H_0^1(\mathcal{O})$-valued predictable $L^2 (\cO)$-continuous processes
$(u_t)_{t\geq0}$ which satisfy
\[
E\sup_{t\in[0,T]}\Vert u_t\Vert^2+E\int
_0^T\mathcal {E}(u_t)\,dt<+\infty.
\]
It is the natural space for solutions.

The space of test functions is denoted by $\mathcal{D}=\mathcal{C}%
_{c}^{\infty} (\R^+ )\otimes\mathcal{C}_c^2 (\cO)$, where
$\mathcal
{C}%
_{c}^{\infty} (\R^+ )$ is the space of all real-valued infinitely
differentiable functions with compact support in $\mathbb{R}^+$ and
$\mathcal{C}_c^2 (\cO)$ is the set
of $C^2$-functions with compact support in~$\cO$.

%the set of test functions, which is the tensor product of the two
%Dirichlet structures $H^1([0,T])$ which is the standard Sobolev
%space on $[0,T]$ and $H_0^1(\mathcal{O})$, i.e.
%$\varphi\in\mathcal{D}$ if and only if\begin{enumerate}
% \item$\varphi\in
% L^2([0,T]\times\mathcal{O})$;
% \item for almost all
% $x\in\mathcal{O}$, $\varphi(\cdot,x)\in
% H^1([0,T])$;
% \item for almost all
% $t\in[0,T]$, $\varphi(t,\cdot)\in
% H_0^1(\mathcal{O})$;
% \item
% $\int_0^T\int_{\mathcal{O}}|\varphi_t|^2+|\partial_t\varphi_t|^2+|
% \end{enumerate}
Heuristically, a pair $(u,\nu)$ is a solution of the obstacle
problem for (\ref{SPDE}) with Dirichlet boundary condition if we have
the following:
\begin{longlist}[(1)]
\item[(1)]$u\in\mathcal{H}_T$ and $u(t,x)\geq S(t,x), dP\otimes
dt\otimes
dx$-a.e. and $u_0(x)=\xi, dP\otimes dx$-a.e.;
\item[(2)]$\nu$ is a random measure defined on
$[0,T)\times\mathcal{O}$;
\item[(3)] the following relation holds almost surely, for all
$t\in[0,T]$ and $\forall\varphi\in\mathcal{D}$,
\begin{eqnarray*}
&&(u_t,\varphi_t)-(\xi,\varphi_0)-\int
_0^t(u_s,\partial_s
\varphi_s)\,ds+\int_0^t
\mathcal{E}(u_s,\varphi_s)\,ds\\
&&\quad{}+\sum
_{i=1}^d\int_0^t
\bigl(g^i_s(u_s,\nabla u_s),
\partial_i\varphi_s\bigr)\,ds
\\
&&\qquad=\int_0^t\bigl(f_s(u_s,
\nabla u_s),\varphi_s\bigr)\,ds+\sum
_{j=1}^{+\infty
}\int_0^t
\bigl(h^j_s(u_s,\nabla u_s),
\varphi_s\bigr)\,dB^j_s\\
&&\qquad\quad{}+\int
_0^t\int_{\mathcal
{O}}
\varphi_s(x)\nu(dx,ds);
\end{eqnarray*}
\item[(4)]
\[
\int_0^T\int_{\mathcal{O}}
\bigl(u(s,x)-S(s,x)\bigr)\nu(dx,ds)=0\qquad \mbox{a.s.}
\]
\end{longlist}

But, the random measure, which in some sense obliges the solution to
stay above the barrier, is a local time so, in general, it is not
absolutely continuous w.r.t. the Lebesgue measure. As a consequence, for
example, the condition\looseness=-1
\[
\int_0^T\int_\mathcal{O}
\bigl(u(s,x)-S(s,x)\bigr)\nu(dx\,ds)=0
\]\looseness=0
makes no
sense. Hence, we need to consider a precise version of $u$ and $S$
defined $\nu$-almost surely.

In order to tackle this difficulty, we introduce in the next section
the notions of parabolic capacity on
$[0,T]\times\mathcal{O}$ and a quasi-continuous version of functions
introduced by Michel Pierre in several works (see, e.g., \cite
{Pierre,PIERRER}). Let us remark that these tools were also
used by Klimsiak \cite{Klimsiak} to get a probabilistic
interpretation to semilinear PDEs with obstacle.

Finally and to end this section, we give an important example of
stochastic noise which is covered by our framework:
%
%ex1 #&#
\begin{example}
\label{mainex} Let $W$ be a
noise white in time and colored in space, defined on a standard
filtered probability
space $  ( \Omega,\cF, (\cF_t)_{t \geq0}, P  )$ whose
covariance function is given by
\[
\forall s,t\in\R^+, \forall x,y\in\cO\qquad E\bigl[\dot{W} (x,s)\dot {W}(y,t)\bigr]=
\delta(t-s)k(x,y),
\]
where $k\dvtx \cO\times\cO\mapsto\R^+$ is a symmetric and measurable
function.

Consider the following SPDE driven by $W$:
%
%e7 #&#
%e8 #&#
\begin{eqnarray}
\label{eW} du_t (x)& =& \Biggl( \sum
_{i,j=1}^d \partial_i a_{i,j}(x)
\partial_j u_t (x) +f\bigl(t,x,u_t (x),\nabla
u_t (x)\bigr)\nonumber\\
&&\hspace*{77pt}{} + \sum_{i=1}^d
\partial_i g_i\bigl(t,x,u_t(x),\nabla
u_t (x)\bigr) \Biggr) \,dt
\\
&&{} + \tilde{h} \bigl(t,x,u_t(x),\nabla u_t(x)\bigr)
W(dt,x), \nonumber%
\end{eqnarray}
where $f$ and $g$ are as above and $\tilde{h}$ is a random real-valued
function.

We assume that the covariance function $k$ defines a trace class operator
denoted by $K$ in $L^2 (\mathcal{O})$. It is well known (see \cite
{RiezNagy}) that there exists an orthogonal
basis $(e_i )_{i\in\Ne}$ of $L^2 (\cO)$ consisting of eigenfunctions of
$K$ with corresponding eigenvalues $(\lambda_i )_{i\in\Ne}$ such that
\[
\sum_{i=1}^{+\infty} \lambda_i <+
\infty
\]
and
\[
k(x,y)=\sum_{i=1}^{+\infty}
\lambda_i e_i (x)e_i (y).
\]
It is also well known that there exists a sequence $((B^i
(t))_{t\geq0})_{i\in\Ne}$ of
independent standard Brownian motions such that
\[
W(dt, \cdot)=\sum_{i=1}^{+\infty}
\lambda_i^{1/2} e_i B^i (dt),
\]
so that equation \eqref{eW} is equivalent to \eqref{SPDE} with
$h=(h_i)_{i\in\Ne}$ where
\[
\forall i\in\Ne\qquad h_i (s,x, y,z)=\sqrt{\lambda_i}\tilde
{h}(s,x,y,z)e_i (x).
\]
%
%If $v\in L^2 (\cO)$, we denote
%dxdy.\]
Assume as in \cite{SW} that for all $i\in\Ne$, $\| e_i \|_{\infty}
<+\infty$ and
\[
\sum_{i=1}^{+\infty} \lambda_i \|
e_i \|_{\infty}^2 <+\infty.
\]
Since
\begin{eqnarray*}
&&\bigl(\bigl| h( t,\omega,x,y,z) -h( t,\omega,x,y{\p},z{\p}) \bigr|^2
\bigr)^{{1}/{2}}\\
&&\qquad\leq \Biggl( \sum_{i=1}^{+\infty}
\lambda_i \| e_i \|_{\infty}^2 \Biggr)
\bigl\llvert \tilde{h} (t,x,y,z)-\tilde{h}(t,x,y{\p },z{\p })\bigr
\rrvert^2,
\end{eqnarray*}
$h$ satisfies the Lipschitz hypothesis \textup{(H)-(3)} if $\tilde{h}$
satisfies a similar Lipschitz hypothesis.
\end{example}

%s3 #&#
\section{Parabolic potential analysis}{\label{capacity}}
%s3.1 #&#
\subsection{Parabolic capacity and potentials}
In this section we will recall some important definitions and
results concerning the obstacle problem for parabolic PDE in
\cite{Pierre} and \cite{PIERRER}.

$\mathcal{K}$ denotes $L^\infty([0,T];L^2(\mathcal{O}))\cap
L^2([0,T];H_0^1(\mathcal{O}))$ equipped with the norm
\begin{eqnarray*}
\Vert v\Vert^2_\mathcal{K}&=&\Vert v\Vert^2_{L^\infty([0,T];L^2(\mathcal{O}))}+
\Vert v\Vert^2_{L^2([0,T];H_0^1(\mathcal{O}))}
\\
&=&\sup_{t\in[0,T[}\Vert v_t\Vert^2 +\int
_0^T \bigl( \Vert v_t
\Vert^2 +\mathcal{E}(v_t) \bigr) \,dt.
\end{eqnarray*}
$\mathcal{C}$ denotes the space of continuous
functions on compact support in $[0,T[\times\mathcal{O}$ and,
finally,
\[
\mathcal{W}=\biggl\{\varphi\in L^2\bigl([0,T];H_0^1(
\mathcal{O})\bigr); \frac
{\partial
\varphi}{\partial t}\in L^2\bigl([0,T];H^{-1}(
\mathcal{O})\bigr)\biggr\},
\]
endowed with the
norm $\Vert \varphi\Vert^2_{\mathcal{W}}=\Vert
\varphi\Vert^2_{L^2([0,T];H_0^1(\mathcal{O}))}+\Vert
\frac{\partial
\varphi}{\partial t}\Vert^2_{L^2([0,T];H^{-1}(\mathcal{O}))}$.

It is known (see \cite{LionsMagenes}) that $\mathcal{W}$ is
continuously embedded in $C([0,T]; L^2 (\cO))$, the set of $L^2 (\cO
)$-valued continuous functions on $[0,T]$. So without ambiguity, we
will also consider
$\mathcal{W}_T=\{\varphi\in\mathcal{W};\varphi(T)=0\}$,
$\mathcal{W}^+=\{\varphi\in\mathcal{W};\varphi\geq0\}$,
$\mathcal{W}_T^+=\mathcal{W}_T\cap\mathcal{W}^+$.

We now introduce the notion of parabolic potentials and regular
measures which permit to define the parabolic capacity.
%
%de1 #&#
\begin{definition}
An element $v\in\mathcal{K}$ is said to be a \textit{parabolic potential}
if it satisfies
\[
\forall\varphi\in\mathcal{W}_T^+\qquad \int_0^T-
\biggl(\frac{\partial\varphi_t}{\partial
t},v_t\biggr)\,dt+\int_0^T
\mathcal{E}(\varphi_t,v_t)\,dt\geq0.
\]
We denote by $\mathcal{P}$ the set of all parabolic potentials.
\end{definition}
The next representation property is crucial:
%
%pr1 #&#
\begin{proposition}[(Proposition 1.1 in \cite{PIERRER})]\label{presentation}
Let $v\in\mathcal{P}$, then there exists a unique positive Radon
measure on $[0,T[\times\mathcal{O}$, denoted by $\nu^v$, such that
\[
\forall\varphi\in\mathcal{W}_T\cap\mathcal{C}\qquad \int
_0^T\biggl(-\frac
{\partial
\varphi_t}{\partial t}, v_t
\biggr)\,dt+\int_0^T\mathcal{E}(
\varphi_t,v_t)\,dt=\int_0^T
\int_\mathcal{O}\varphi(t,x)\,d\nu^v.
\]
Moreover, $v$ admits a right-continuous (resp., left-continuous)
version $\hat{v}$  (resp., $\bar{v})\dvtx [0,T]\mapsto L^2
(\cO)$.

Such a Radon measure $\nu^v$ is called \emph{a regular measure} and
we write
\[
\nu^v =\frac{\partial v}{\partial t}+Av.
\]
\end{proposition}
%
%re2 #&#
\begin{remark} As a consequence, we can also define for all $v\in
\mathcal{P}$,
\[
v_T =\lim_{t\uparrow T}\bar{v}_t \in L^2 (
\cO).
\]
\end{remark}
%
%$\nu$ on $[0,T]\times\mathcal{O}$ is called \textbf{a regular
%measure}, if there exists $v\in\mathcal{P}$ such that
% $\nu=\frac{\partial v}{\partial t}+Av$ in the sense that
% $\forall\varphi\in\mathcal{W}_T\cap\mathcal{C}$, $$\int_0^T(-\frac{

%de2 #&#
\begin{definition}
Let $K\subset[0,T[\times\mathcal{O}$ be compact; $v\in\mathcal{P}$
is said to be $\nu$-superior than 1 on $K$, if there
exists a sequence $v_n\in\mathcal{P}$ with $v_n\geq1$ a.e. on a
neighborhood of $K$ converging to $v$ in
$L^2([0,T];H_0^1(\mathcal{O}))$.
\end{definition}
We denote
\[
\mathscr{S}_K=\{v\in\mathcal{P}; v\mbox{ is }\nu\mbox{-superior to 1 on }K\}.
\]

%pr2 #&#
\begin{proposition}[(Proposition 2.1 in \cite{PIERRER})]
Let $K\subset[0,T[\times\mathcal{O}$ be compact, then $\mathscr{S}_K$
admits a smallest $v_K\in\mathcal{P}$ and the measure $\nu^v_K$
whose support is in $K$ satisfies
\[
\int_0^T\int_\mathcal{O}\,d
\nu^v_K=\inf_{v\in\mathcal{P}}\biggl\{\int
_0^T\int_\mathcal{O}\,d
\nu^v; v\in\mathscr{S}_K\biggr\}.
\]
\end{proposition}
%
%de3 #&#
\begin{definition}[(Parabolic capacity)]
\begin{itemize}
\item Let $K\subset[0,T[\times\mathcal{O}$ be compact, and we
define
$\operatorname{cap}(K)=\int_0^T\int_\mathcal{O}\,d\nu^v_K$.
\item Let $O\subset
[0,T[\times\mathcal{O}$ be open, and we define $\operatorname{cap}(O)=\sup\{
\operatorname{cap}(K); K\subset O\ \mathrm{compact}\}$.\vadjust{\goodbreak}
\item {\spaceskip=0.19em plus 0.05em minus 0.02em For any Borelian
$E\subset[0,T[\times\mathcal{O}$, we define $\operatorname{cap}(E)=\inf\{\operatorname{cap}(O); O\supset E\
\mathrm{open}\}$.}
\end{itemize}
\end{definition}

%de4 #&#
\begin{definition}A property is said to hold quasi-everywhere (in
short,~q.e.)
if it holds outside a set of null capacity.
\end{definition}
%
%de5 #&#
\begin{definition}[(Quasi-continuity)]
A function $u\dvtx [0,T[\times\mathcal{O}\rightarrow\mathbb{R}$
is called quasi-continuous, if there exists a decreasing sequence of open
subsets $O_n$ of $[0,T[\times\mathcal{O}$ with the following:
\begin{longlist}[(1)]
\item[(1)] for all $n$, the restriction of $u_n$ to the complement of $O_n$ is
continuous;
\item[(2)]$\lim_{n\rightarrow+\infty}\operatorname{cap} (O_n)=0$.
\end{longlist}
We say that $u$ admits a quasi-continuous version, if there exists
$\tilde{u}$ quasi-continuous such that $\tilde{u}=u$ a.e.
\end{definition}
The next proposition, whose proof may be found in \cite{Pierre} or
\cite
{PIERRER}, shall play an important role in the sequel:
%
%pr3 #&#
\begin{proposition}\label{Versiont} Let $K\subset\cO$ be a compact
set, then $\forall t\in[0,T[$,
\[
\operatorname{cap} \bigl(\{ t\}\times K\bigr)=\lambda_d (K),
\]
where $\lambda_d$ is the Lebesgue measure on $\cO$.

As a consequence, if $u\dvtx [0,T[\times\cO\rightarrow\R$ is a map
defined quasi-everywhere, then it defines uniquely a map from $[0,T[$
into $L^2 (\cO)$.
In other words, for any $t\in[0,T[$, $u_t$ is defined without any
ambiguity as an element in $L^2 (\cO)$.
Moreover, if $u\in\mathcal{P}$, it admits version $\bar{u}$ which is
left continuous on $[0,T]$ with values in $L^2 (\cO)$ so that $u_T
=\bar{u}_{T^-}$ is also defined without ambiguity.
\end{proposition}
%
%re3 #&#
\begin{remark} The previous proposition applies if, for example, $u$ is
quasi-continuous.
\end{remark}

%pr4 #&#
\begin{proposition}[(Theorem III.1 in \cite{PIERRER})]\label{approximation}
If $\varphi\in\mathcal{W}$, then it admits a unique quasi-continuous
version that we denote by $\tilde{\varphi}$. Moreover, for all $v\in
\mathcal{P}$,
the following relation holds:
\[
\int_{[0,T[\times\cO} \tilde{\varphi}\,d\nu^v =\int
_0^T ( -\partial_t \varphi,v )+
\mathcal{E}(\varphi,v) \,dt + ( \varphi_T,v_{T} ).
\]
\end{proposition}

%s3.2 #&#
\subsection{Applications to PDEs with obstacle}
For any function $\psi\dvtx [0,T[\times\mathcal{O}\rightarrow\bbR$ and
$u_0\in L^2(\mathcal{O})$, following M. Pierre
\cite{Pierre,PIERRER}, F. Mignot and J.P. Puel \cite{MignotPuel}, we
define
%
%e9 #&#
\begin{equation}
\label{kappa}\kappa(\psi,u_0)= \operatorname{ess} \inf \bigl\{u\in\mathcal{P}; u\geq
\psi\mbox{ a.e.}, u(0)\geq u_0\bigr\}.
\end{equation}
This lower bound exists and is an element in
$\mathcal{P}$. Moreover, when $\psi$ is \mbox{quasi-continuous},\vadjust{\goodbreak} this
potential is the solution of the following reflected
problem:
\begin{eqnarray*}
\kappa\in\mathcal{P}, \kappa\geq\psi,\qquad \frac{\partial\kappa}{\partial t}+A\kappa=0\qquad \mbox{on } \{u>\psi\},\qquad
\kappa(0)=u_0.
\end{eqnarray*}
Mignot and Puel have proved in
\cite{MignotPuel} that $\kappa(\psi,u_0)$ is the limit [increasingly
and weakly in $L^2 ([0,T]; H^1_0 (\cO))$] when $\varepsilon$ tends to $0$
of the solution of the following penalized equation:
\begin{eqnarray*}
u_\varepsilon\in\mathcal{W}, \qquad u_\varepsilon(0)=u_0,\qquad
\frac
{\partial u_\varepsilon}{\partial t}+ Au_\varepsilon-\frac{(u_\varepsilon-\psi)^-}{\varepsilon}=0.
\end{eqnarray*}
Let us point out that they obtain this result in the more general
case where $\psi$ is only measurable from $[0,T[$ into
$L^2(\mathcal{O})$.

For given $f\in L^2([0,T];H^{-1}(\mathcal{O}))$,
we denote by $\kappa_{u_0}^f$ the solution of the following
problem:
\[
\kappa\in\mathcal{W},\qquad \kappa(0)=u_0,\qquad \frac{\partial\kappa}{\partial t}+A\kappa=f.
\]
The
next theorem ensures existence and uniqueness of the
solution of parabolic PDE with obstacle; it is proved in
\cite{Pierre}, Theorem 1.1. The proof is based on a regularization
argument of the obstacle, using the results of \cite{CT75}.

%th2 #&#
\begin{theorem}
Let $\psi\dvtx [0,T[\times\mathcal{O}\rightarrow\mathbb{R}$ be
quasi-continuous, suppose that there exists $\zeta\in\mathcal{P}$
with $|\psi|\leq\zeta$ a.e., $f\in L^2([0,T];H^{-1}(\mathcal{O}))$,
and the initial value $u_0\in L^2(\mathcal{O})$ with
$u_0\geq\psi(0)$, then there exists a unique
$u\in\kappa_{u_0}^f+\mathcal{P}$ quasi-continuous such that
\begin{eqnarray*}
u(0)=u_0, \tilde{u}\geq\psi,\mbox{ q.e.};\qquad \int_0^T
\int_\mathcal{O}(\tilde{u}-\tilde{\psi})\,d\nu^{u-\kappa_{u_0}^f}=0.
\end{eqnarray*}
\end{theorem}

%Then $u\in\mathcal{W}\cap(\kappa_{u_0}^0+\mathcal{P})$ and we have
%$$\frac{\partial u}{\partial t}+Au\leq(\frac{\partial\psi}{\partial
%t}+A\psi)^+ in L^2(0,T;H^{-1}(\mathcal{O})).$$
%Moreover, $u$ is unique element of $\mathcal{W}$ such that
%$$u(0)=u_0\vee\psi(0), \int_0^T(\frac{\partial u}{\partial t}+Au,u-
%Then one approximate the quasi-continuous obstacle by the elements
%in $\breve{\mathcal{W}}$.
%
We end this section by a convergence lemma which plays an important
role in our approach (Lemma 3.8 in \cite{PIERRER}):
%
%le1 #&#
\begin{lemma}\label{convergemeas}
If $(v^n)_n\in\mathcal{P}$ is a bounded sequence in $\mathcal{K}$ and
converges weakly to~$v$ in $L^2([0,T];H_0^1(\mathcal{O})), and $ if
$u$ is
a quasi-continuous function and $|u|$ is bounded by a element in
$\mathcal{P}$, then
\[
\lim_{n\rightarrow+\infty}\int_0^T\int
_\mathcal{O}u\,d\nu^{v^n}=\int_0^T
\int_\mathcal{O}u\,d\nu^{v}.
\]
\end{lemma}
%
%re4 #&#
\begin{remark}For the more general case one can see \cite{PIERRER},
Lemma 3.8.
\end{remark}

%s4 #&#
\section{Quasi-continuity of the solution of SPDE without
obstacle}\label{quasi-contSPDE}
As a consequence of well-known results (see, e.g., \cite{DenisStoica},
Theorem 8), we know that under Assumptions \ref{assH} and \ref{assI},
SPDE (\ref{SPDE}) with zero Dirichlet boundary condition admits a
unique solution in $\mathcal{H}_T$ (for the definition of solution see,
e.g., Definition~1 in \cite{DenisStoica}); we denote it by
$\mathcal{U}(\xi,f,g,h)$.
The main theorem of this section is
the following:\vadjust{\goodbreak}
%
%th3 #&#
\begin{theorem}\label{mainquasicontinuity} Under
Assumptions \textup{\ref{assH}} and \textup{\ref{assI}}, $u=\mathcal{U}(\xi,f,g,h)$ the
solution of SPDE (\ref{SPDE}) admits a quasi-continuous version
denoted by $\tilde{u}$, that is,
$u=\tilde{u}$ $dP\otimes dt\otimes dx $-a.e. and for almost all $w\in
\Omega$,
$(t,x)\rightarrow\tilde{u}_t
(w,x)$ is quasi-continuous.
\end{theorem}
Before giving the proof of this theorem, we need the following lemmas.
The first one is proved in \cite{PIERRER}, Lemma 3.3:
%
%le2 #&#
\begin{lemma}\label{cap}There exists $C>0$ such that, for all open set
$\vartheta\subset[0,T[\times\mathcal{O}$ and $v\in\mathcal{P}$ with
$v\geq1$ a.e. on $\vartheta$,
\[
\operatorname{cap}\vartheta\leq C\Vert v\Vert^2_{\mathcal{K}}.
\]
\end{lemma}
Let $\kappa=\kappa(u,u^+(0))$ be defined by relation \eqref{kappa}. One
has to note that $\kappa$ is a random function. From now on, we always
take for $\kappa$ the following measurable version
\[
\kappa=\sup_n v^n,
\]
where $(v^n)_n$ is the nondecreasing sequence of random functions given by
%
%e10 #&#
%e11 #&#
\begin{eqnarray}
\label{eq1} \cases{ %
\displaystyle\frac{\partial v_t^n}{\partial t}=Lv_t^n+n
\bigl(v_t^n-u_t\bigr)^-,
\vspace*{2pt}\cr
 \displaystyle v^n_0=u^+(0). }
\end{eqnarray}
Using the results recalled in Section \ref{capacity}, we know that
for almost all $w\in\Omega$, $v^n (w)$ converges weakly to
$v(w)=\kappa
(u(w),u^+(0)(w))$ in
$L^2([0,T];H_0^1(\mathcal{O}))$ and that $v\geq u$.
%
%le3 #&#
\begin{lemma}\label{estimoftau} We have the following estimate:
\begin{eqnarray*}
E\Vert \kappa\Vert_{\mathcal{K}}^2 \leq C \biggl(E\bigl\Vert
u_0^+\bigr\Vert^2+E\Vert u_0\Vert^2+E
\int_0^T\bigl\Vert f_t^0
\bigr\Vert^2+\bigl\Vert \bigl|g_t^0\bigr|\bigr\Vert^2+\bigr\Vert
\bigl|h_t^0\bigr|\bigr\Vert^2\,dt \biggr),
\end{eqnarray*}
where $C$ is a constant depending only on the structure constants of
the equation.
\end{lemma}
\begin{pf} All along this proof, we shall denote by $C$ or
$C_{\varepsilon} $ some constant which may change from line to line.

The following estimate for the solution of the SPDE we consider is well known:
%
%e12 #&#
\begin{eqnarray}
\label{estimateu}&&E\sup_{t\in[0,T]}\Vert u_t\Vert^2+E
\int_0^T\mathcal{E}(u_t)\,dt
\nonumber
\\[-8pt]
\\[-8pt]
\nonumber
&&\qquad\leq CE
\biggl(\Vert u_0\Vert^2+\int_0^T
\bigl(\bigl\Vert f_t^0\bigr\Vert^2+\bigl\Vert
\bigl|g_t^0\bigr|\bigr\Vert^2+\bigl\Vert \bigl|h_t^0\bigr|
\bigr\Vert^2\bigr)\,dt\biggr),
\end{eqnarray}
where $C$ is a constant depending only on the structure constants of
the equation.

Consider the approximation $(v^n )_n$ defined by \eqref{eq1},
$P$-almost surely, it converges weakly to $v=\kappa(u,u^+(0))$ in
$L^2([0,T];H_0^1(\mathcal{O}))$.

We remark that $v^n-u$ satisfies
the following
equation:
\begin{eqnarray*}
&& d\bigl(v^n_t-u_t\bigr)+A
\bigl(v_t^n-u_t\bigr)\,dt\\
&&\qquad=-f_t(u_t,
\nabla u_t)\,dt-\sum_{i=1}^d
\partial_ig^i_t(u_t,\nabla
u_t)\,dt\\
&&\qquad\quad{}-\sum_{j=1}^{+\infty}h^j_t(u_t,
\nabla u_t)\,dB^j_t+n\bigl(v_t^n-u_t
\bigr)^-\,dt.
\end{eqnarray*}
Applying It\^o's
formula to $(v^n-u)^2$ (see Lemma 7 in \cite{DMS05}), we have almost
surely, for all $t\in[0,T]$,
%
%e13 #&#
\begin{eqnarray}
\label{original}
&&\bigl\Vert v_t^n-u_t
\bigr\Vert^2+2\int_0^t\mathcal{E}
\bigl(v_s^n-u_s\bigr)\,ds\nonumber\\
&&\qquad=\bigl\Vert
u^+_0-u_0\bigr\Vert^2-2\int_0^t
\bigl(v_s^n-u_s,f_s(u_s,
\nabla u_s)\bigr)\,ds
\nonumber
\\[-8pt]
\\[-8pt]
\nonumber
&&\quad\qquad{}+2\sum_{i=1}^d\int_0^t
\bigl(\partial_i\bigl(v_s^n-u_s
\bigr),g^i_s(u_s,\nabla u_s)
\bigr)\,ds+\int_0^t\bigl\Vert \bigl|h_s(u_s,
\nabla u_s)\bigr|\bigr\Vert^2\,ds
\\
&&\qquad\quad{}-2\sum_{j=1}^{+\infty}\int_0^t
\bigl(v_s^n-u_s,h^j_s(u_s,
\nabla u_s)\bigr)\,dB^j_s+2\int
_0^t\bigl(n\bigl(v_s^n-u_s
\bigr)^-,v_s^n-u_s\bigr)\,ds.\nonumber
\end{eqnarray}
The
last term in the right member of (\ref{original}) is obviously
nonpositive, so
%
%e14 #&#
\begin{eqnarray}
\label{originalineq}&&\bigl\Vert v_t^n-u_t
\bigr\Vert^2+2\int_0^t\mathcal{E}
\bigl(v_s^n-u_s\bigr)\,ds\nonumber\\
&&\qquad\leq\bigl\Vert
u^+_0-u_0\bigr\Vert^2-2\int_0^t
\bigl(v_s^n-u_s,f_s(u_s,
\nabla u_s)\bigr)\,ds
\nonumber
\\[-8pt]
\\[-8pt]
\nonumber
&&\qquad\quad{}+\int_0^t\bigl\Vert \bigl|h_s(u_s,
\nabla u_s)\bigr|\bigr\Vert^2\,ds+2\sum_{i=1}^d
\int_0^t\bigl(\partial_i
\bigl(v_s^n-u_s\bigr),g^i_s(u_s,
\nabla u_s)\bigr)\,ds
\\
&&\qquad\quad{}-2\sum_{j=1}^{+\infty}\int_0^t
\bigl(v_s^n-u_s,h^j_s(u_s,
\nabla u_s)\bigr)\,dB^j_s\qquad \mbox{a.s.}\nonumber
\end{eqnarray}

Then taking
expectation and using Cauchy--Schwarz's inequality, we get
\begin{eqnarray*}
&&E\bigl\Vert v_t^n-u_t\bigr\Vert^2+
\biggl(2-\frac{\varepsilon
}{\lambda}\biggr)E\int_0^t
\mathcal{E}\bigl(v_s^n-u_s\bigr)\,ds\\
&&\qquad\leq E
\bigl\Vert u_0^+-u_0\bigr\Vert^2+E\int
_0^t\bigl\Vert v_s^n-u_s
\bigr\Vert^2\,ds
\\
&&\qquad\quad{}+E\int_0^t\bigl\Vert f_s(u_s,
\nabla u_s)\bigr\Vert^2\,ds+C_\varepsilon E\int
_0^t\bigl\Vert \bigl|g_s(u_s,
\nabla u_s)\bigr|\bigr\Vert^2\,ds\\
&&\qquad\quad{}+E\int_0^t
\bigl\Vert \bigl|h_s(u_s,\nabla u_s)\bigr|
\bigr\Vert^2\,ds.
\end{eqnarray*}
Therefore, by using the Lipschitz conditions on the coefficients, we
have
\begin{eqnarray*}
&&E\bigl\Vert v_t^n-u_t\bigr\Vert^2+
\biggl(2-\frac{\varepsilon
}{\lambda}\biggr)E\int_0^t
\mathcal{E}\bigl(v_s^n-u_s\bigr)\,ds\\
&&\qquad\leq E
\bigl\Vert u_0^+-u_0\bigr\Vert^2+E\int
_0^t\bigl\Vert v_s^n-u_s
\bigr\Vert^2\,ds
\\
&&\qquad\quad{}+CE\int_0^t\bigl(\bigl\Vert f_s^0
\bigr\Vert^2+\bigl\Vert \bigl|g_s^0\bigr|\bigr\Vert^2+
\bigl\Vert \bigl|h_s^0\bigr|\bigr\Vert^2\bigr)\,ds+CE\int
_0^t\Vert u_s\Vert^2\,ds\\
&&\qquad\quad{}+
\biggl(\frac{C}{\lambda}+\frac{\alpha}{\lambda}+\frac{\beta^2}{\lambda
}\biggr)E\int
_0^t\mathcal{E}(u_s)\,ds.
\end{eqnarray*}
Combining with (\ref{estimateu}), this yields
\begin{eqnarray*}
&&E\bigl\Vert v_t^n-u_t\bigr\Vert^2+
\biggl(2-\frac{\varepsilon
}{\lambda}\biggr)E\int_0^t
\mathcal{E}\bigl(v_s^n-u_s\bigr)\,ds\\
&&\qquad\leq E
\bigl\Vert u_0^+-u_0\bigr\Vert^2+E\int
_0^t\bigl\Vert v_s^n-u_s
\bigr\Vert^2\,ds
\\
&&\quad\qquad{}+ CE\biggl(\Vert u_0\Vert^2+\int_0^T
\bigl(\bigl\Vert f_t^0\bigl\Vert^2+\bigl\Vert
\bigl|g_t^0\bigr|\bigr\Vert^2+\bigl\Vert\bigl |h_t^0\bigr|
\bigr\Vert^2\bigr)\,dt\biggr).
\end{eqnarray*}
We take now $\varepsilon$
small enough such that $(2-\frac{\varepsilon}{\lambda})>0$, then, with
Gronwall's lemma, we obtain for each $t\in[0,T]$
\begin{eqnarray*}
&&E\bigl\Vert v_t^n-u_t\bigr\Vert^2\\
&&\qquad\leq
Ce^{c'T}\biggl(E\bigl\Vert u_0^+-u_0
\bigr\Vert^2+E\Vert u_0\Vert^2+E\int
_0^T\bigl\Vert f_t^0
\bigr\Vert^2+\bigl\Vert \bigl|g_t^0\bigr|\bigr\Vert^2+\bigl\Vert
\bigr|h_t^0\bigr|\bigr\Vert^2\,dt\biggr).
\end{eqnarray*}
As we a priori know that $P$-almost surely, $(v^n)_n$ tends to $\kappa$
strongly in $L^2 ([0,T]\times\cO)$, the previous estimate yields,
thanks to the dominated convergence theorem, that $(v^n )_n$ converges to
$\kappa$ strongly in $L^2 (\Omega\times[0,T]\times\cO)$ and
\begin{eqnarray*}
&&\sup_{t\in[0,T]}E\Vert \kappa_t-u_t\Vert^2
\\
&&\qquad\leq Ce^{c'T}\biggl(E\bigl\Vert u_0^+-u_0
\bigr\Vert^2+E\Vert u_0\Vert^2+E\int
_0^T\bigl\Vert f_t^0
\bigr\Vert^2+\bigl\Vert \bigl|g_t^0\bigr|\bigr\Vert^2+\bigl\Vert
\bigl|h_t^0\bigr|\bigr\Vert^2\,dt\biggr).
\end{eqnarray*}
Moreover, as $(v^n)_n$ tends to $\kappa$ weakly in $L^2 ([0,T]; H^1_0
(\cO))$ $P$-almost-surely, we have for all $t\in[0,T]$,
\begin{eqnarray*}
&&\hspace*{-6pt}E \int_0^T \mathcal{E}(\kappa_s
-u_s)\,ds\\
&&\hspace*{-6pt}\qquad\leq \liminf_n E\int_0^T
\mathcal{E}\bigl(v_s^n-u_s\bigr)\,ds
\\
&&\hspace*{-6pt}\qquad\leq TCe^{c'T}\biggl(E\bigl\Vert u_0^+-u_0
\bigr\Vert^2+E\Vert u_0\Vert^2+E\int
_0^T\bigl\Vert f_t^0
\bigr\Vert^2+\bigl\Vert \bigl|g_t^0\bigr|\bigr\Vert^2+\bigl\Vert
\bigl|h_t^0\bigr|\bigr\Vert^2\,dt\biggr).
\end{eqnarray*}

Let us now study the stochastic term in \eqref{originalineq}. Let us
define the martingales
\[
M^n_t =\sum_{j=1}^{+\infty}
\int_0^t\bigl(v_s^n-u_s,h^j_s
\bigr)\,dB^j_s \quad\mbox{and} \quad M_t=\sum
_{j=1}^{+\infty}\int_0^t
\bigl(\kappa_s-u_s,h^j_s
\bigr)\,dB^j_s.
\]
Then
\begin{eqnarray*}
&&E\bigl[\bigl|M^n_T -M_T\bigr|^2\bigr]\\
&&\qquad=E
\int_0^T \sum_{j=1}^{+\infty}
\bigl(\kappa_s -v^n_s,h_s
\bigr)^2 \,ds\\
&&\qquad\leq E\int_0^T\bigl\Vert
\kappa_s-v^n_s\bigr\Vert^2\bigl\Vert
\bigl|h_s\bigr|\bigr\Vert^2\,ds.
\end{eqnarray*}
Using the strong convergence of $(v^n )_n$ to $\kappa$, we conclude
that $(M^n)_n$ tends to $M$ in the $L^2$ sense.
Passing to the limit in \eqref{originalineq}, we get, almost surely,
for all $t\in[0,T]$,
\begin{eqnarray*}
&&\Vert \kappa_t-u_t\Vert^2+ 2
\int_0^t\mathcal {E}(\kappa_s
-u_s)\,ds\\
&&\qquad\leq\bigl\Vert u^+_0-u_0
\bigr\Vert^2-2\int_0^t\bigl(
\kappa_s-u_s,f_s(u_s,\nabla
u_s)\bigr)\,ds
\\
&&\qquad\quad{}+2\sum_{i=1}^d\int_0^t
\bigl(\partial_i(\kappa_s -u_s),g^i_s(u_s,
\nabla u_s)\bigr)\,ds
\\
&&\qquad\quad{}- 2\sum_{j=1}^{+\infty}\int_0^t
\bigl(\kappa_s-u_s,h^j_s(u_s,
\nabla u_s)\bigr)\,dB^j_s\\
&&\qquad\quad{}+\int
_0^t\bigl\Vert \bigl|h_s(u_s,
\nabla u_s)\bigr|\bigr\Vert^2\,ds.
\end{eqnarray*}

As a consequence of Burkholder--Davies--Gundy's inequalities, we get
\begin{eqnarray*}
&&E\sup_{t\in[0,T]}\biggl|\sum_{j=1}^{+\infty}\int
_0^t\bigl(\kappa_s-u_s,h^j_s
(u_s,\nabla u_s )\bigr)\,dB^j_s\biggr|\\
&&\qquad\leq CE\Biggl[\int_0^T\sum
_{j=1}^{+\infty}\bigl(\kappa_s-u_s,h^j_s(u_s,
\nabla u_s)\bigr)^2\,ds\Biggr]^{1/2}
\\
&&\qquad\leq CE\Biggl[\int_0^T\sum
_{j=1}^{+\infty}\sup_{t\in[0,T]}\Vert
\kappa_t-u_t\Vert^2\bigl\Vert h^j_s
(u_s,\nabla u_s )\bigr\Vert^2\,ds
\Biggr]^{1/2}
\\
&&\qquad\leq CE\biggl[\sup_{t\in[0,T]}\Vert \kappa_t-u_t\Vert
\biggl(\int_0^T\bigl\Vert \bigl|h_t
(u_t,\nabla u_t )\bigr|\bigr\Vert^2\,dt
\biggr)^{1/2}\biggr]
\\
&&\qquad\leq\varepsilon E\sup_{t\in[0,T]}\Vert \kappa_t -u_t
\Vert^2+C_\varepsilon E\int_0^T
\bigl\Vert \bigl|h_t (u_t,\nabla u_t )\bigr|
\bigr\Vert^2\,dt.
\end{eqnarray*}
By Lipschitz conditions
on $h$ and (\ref{estimateu}) this yields
\begin{eqnarray*}
&& E\sup_{t\in[0,T]}\Biggl|\sum_{j=1}^{+\infty}\int
_0^t\bigl(\kappa_s-u_s,h_s
(u_s,\nabla u_s )\bigr)\,dB_s\Biggr|\\
&&\qquad\leq
\varepsilon E\sup_{t\in[0,T]}\Vert \kappa_t-u_t
\Vert^2\\
&&\qquad\quad{}+C\biggl(E\Vert u_0\Vert^2+E\int_0^T\bigl(\bigl\Vert f_t^0
\bigr\Vert^2+\bigl\Vert\bigl |g_t^0\bigr|\bigr\Vert^2+\bigl\Vert
\bigl|h_t^0\bigr|\bigr\Vert^2\bigr)\,dt\biggr).
\end{eqnarray*}
Hence,
\begin{eqnarray*}
&&(1-\varepsilon)E\sup_{t\in[0,T]}\Vert \kappa_t-u_t
\Vert^2+\biggl(2-\frac{\varepsilon}{\lambda}\biggr)E\int_0^T
\mathcal {E}(\kappa_t-u_t)\,dt\\
&&\qquad\leq C\biggl(E\bigl\Vert
u_0^+-u_0\bigr\Vert^2+E\Vert u_0
\Vert^2+E\int_0^T\bigl\Vert f_t^0
\bigr\Vert^2 +\bigl\Vert \bigl|g_t^0\bigr|\bigr\Vert^2+
\bigl\Vert \bigl|h_t^0\bigr|\bigr\Vert^2\,dt\biggr).
\end{eqnarray*}
We can take $\varepsilon$ small
enough such that $1-\varepsilon>0$ and $2-\frac{\varepsilon}{\lambda}>0$,
hence,
\begin{eqnarray*}
&&E\sup_{t\in[0,T]}\Vert \kappa_t -u_t
\Vert^2+E\int_0^T\mathcal{E}(
\kappa_t-u_t)\,dt\\
&&\qquad\leq C\biggl(E\bigl\Vert u_0^+-u_0
\bigr\Vert^2+E\Vert u_0\Vert^2
+E\int_0^T\bigl\Vert f_t^0
\bigr\Vert^2+\bigl\Vert\bigl |g_t^0\bigr|\bigr\Vert^2+\bigl\Vert
\bigl|h_t^0\bigr|\bigr\Vert^2\,dt\biggr).
\end{eqnarray*}
Then, combining with
(\ref{estimateu}), we get the desired estimate:
\begin{eqnarray*}
&&E\sup_{t\in[0,T]}\Vert \kappa_t\Vert^2+E\int
_0^T\mathcal{E}(\kappa_t)\,dt\\
&&\qquad\leq C
\biggl(E\bigl\Vert u_0^+\bigr\Vert^2+E\Vert u_0
\Vert^2+E\int_0^T\bigl\Vert
f_t^0\bigr\Vert^2+\bigl\Vert \bigl|g_t^0\bigr|
\bigr\Vert^2+\bigl\Vert \bigl|h_t^0\bigr|\bigr\Vert^2\,dt
\biggr).
\end{eqnarray*}
\upqed\end{pf}

\begin{pf*}{Proof of Theorem \ref{mainquasicontinuity}} For simplicity,
we put
\begin{eqnarray*}
f_t (x)&=&f\bigl(t,x,u_t (x),\nabla u_t (x)
\bigr),\qquad g_t (x)=g\bigl(t,x,u_t (x),\nabla u_t
(x)\bigr)\quad\mbox{and} \\
 h_t (x)&=&h\bigl(t,x,u_t (x),
\nabla u_t (x)\bigr).
\end{eqnarray*}
We introduce $(P_t )$ the semi-group associated to operator $A$ and put
for each $n\in\mathbb{N}^*$, $i\in\{ 1,\ldots,d\}$ and each $j\in
\N^*$,
\[
u_0^n=P_{{1}/{n}}u_0,\qquad
f^n=P_{{1}/{n}}f,\qquad g_i^n=P_{{1}/{n}}g_i\quad
\mbox{and}\quad h_j^n=P_{{1}/{n}}h_j.
\]
Then $(u_0^n )_n$ converges to $u_0$ in
$L^2(\Omega;L^2(\mathcal{O}))$, and $(f^n )_n$, $(g^n )_n$ and $(h^n
)_n$ are
sequences of elements in $L^2(\Omega\times[0,T];\mathcal{D}(A))$
which converge, respectively, to~$f$, $g$ and $h$ in
$L^2(\Omega\times[0,T];L^2(\mathcal{O}))$. For all $n\in\mathbb{N}^*$
we define
\begin{eqnarray*}
u_t^n&=&P_tu_0^n+
\int_0^tP_{t-s}f_s^n\,ds+
\sum_{i=1}^d \int_0^tP_{t-s}
\partial_i g_{i,s}^n\,ds+\sum
_{j=1}^{+\infty}\int_0^tP_{t-s}h_{j,s}^n\,dB^j_s
\\
&=&P_{t+({1}/{n})}u_0+\int_0^tP_{t+({1}/{n})-s}f_s\,ds+
\sum_{i=1}^d \int_0^tP_{t+({1}/{n})-s}
\partial_i g_{i,s}\,ds\\
&&{}+\sum_{j=1}^{+\infty
}
\int_0^tP_{t+({1}/{n})-s}h_{j,s}\,dB^j_s.
\end{eqnarray*}
We denote by $G(t,x,s,y)$ the kernel associated to $P_t$, then
\begin{eqnarray*}
u^n(t,x)&=&\int_\cO G\biggl(t+
\frac{1}{n},x,0,y\biggr)u_0(y)\,dy\\
&&{}+\int_0^t
\int_\cO G\biggl(t+\frac{1}{n},x,s,y\biggr)f(s,y)\,dy\,ds
\\
&&{}+\sum_{i=1}^d\int_0^t
\int_\cO G\biggl(t+\frac{1}{n},x,s,y\biggr)
\partial_ig_s^{i}(y)\,dy\,ds\\
&&{}+\sum
_{j=1}^{+\infty
}\int_0^t
\int_\cO G\biggl(t+\frac{1}{n},x,s,y
\biggr)h_s^{j}(y)\,dy\,dB^i_s.
\end{eqnarray*}
But, as $A$
is strictly elliptic, $G$ is uniformly continuous in space--time
variables on any compact away from the diagonal in time (see
Theorem 6
in \cite{Aronson}) and satisfies Gaussian estimates (see Aronson \cite
{Aronson3}); this ensures that for all $n\in\bbN^*$, $u^n$ is
$P$-almost surely continuous in $(t,x)$.

We consider a sequence of random open sets
\[
\vartheta_n=\bigl\{\bigl|u^{n+1}-u^n\bigr|>
\varepsilon_n\bigr\},\qquad \Theta_p=\bigcup
_{n=p}^{+\infty}\vartheta_n.
\]
Let
$\kappa_n=\kappa(\frac{1}{\varepsilon_n }(u^{n+1}-u^n),\frac
{1}{\varepsilon_n
}(u^{n+1}-u^n)^+(0))+\kappa(-\frac{1}{\varepsilon_n
}(u^{n+1}-u^n),\frac
{1}{\varepsilon_n }(u^{n+1}-u^n)^-(0))$, and
from the definition of $\kappa$ and the relation (see \cite{PIERRER})
\[
\kappa\bigl(|v|\bigr)\leq\kappa\bigl(v,v^+(0)\bigr)+\kappa\bigl(-v,v^-(0)\bigr),
\]
we know that
$\kappa_n$ satisfy the conditions of Lemma \ref{cap}, that is,
$\kappa_n\in\mathcal{P}$ et $\kappa_n\geq1$ a.e. on $\vartheta_n$,
thus, we get the following relation:
\[
\operatorname{cap} (\Theta_p )\leq\sum_{n=p}^{+\infty}\operatorname{cap}
(\vartheta_n )\leq\sum_{n=p}^{+\infty}
\Vert \kappa_n\Vert^2_{\mathcal{K}}.
\]

Thus, remarking that $u^{n+1}-u^n =\mathcal{U}(u_0^{n+1}-u_0^n,f^{n+1}-f^n,g^{n+1}-g^n, h^{n+1}-h^n )$, we apply Lemma \ref
{estimoftau} to $\kappa(\frac{1}{\varepsilon_n }(u^{n+1}-u^n),\frac
{1}{\varepsilon_n }(u^{n+1}-u^n)^+(0))$ and\break  $\kappa(-\frac{1}{\varepsilon_n
}(u^{n+1}-u^n),\frac{1}{\varepsilon_n }(u^{n+1}-u^n)^-(0))$ and obtain
\begin{eqnarray*}
&&E\bigl[\operatorname{cap} ( \Theta_p)\bigr]\\
&&\qquad\leq\sum_{n=p}^{+\infty}
E\Vert \kappa_n\Vert^2_{\mathcal{K}}\\
&&\qquad\leq 2C\sum
_{n=p}^{+\infty}\frac{1}{\varepsilon_n^2}\biggl(E\bigl\Vert
u_0^{n+1}-u_0^n
\bigr\Vert^2+E\int_0^T\bigl\Vert
f_t^{n+1}-f_t^n\bigr\Vert^2
+\bigl\Vert \bigl|g_t^{n+1}-g_t^n\bigr|
\bigr\Vert^2\\
&&\hspace*{247pt}{}+\bigl\Vert \bigl|h_t^{n+1}-h_t^n\bigr|
\bigr\Vert^2\,dt\biggr).
\end{eqnarray*}

Then, by extracting a subsequence, we can consider that
\begin{eqnarray*}
&&E\bigl\Vert u_0^{n+1}-u_0^n
\bigr\Vert^2+E\int_0^T\bigl\Vert
f_t^{n+1}-f_t^n\bigr\Vert^2+
\bigl\Vert \bigl|g_t^{n+1}-g_t^n\bigr|
\bigr\Vert^2+\bigl\Vert\bigl |h_t^{n+1}-h_t^n\bigr|
\bigr\Vert^2\,dt\\
&&\qquad\leq\frac{1}{2^n}.
\end{eqnarray*}
Then we take $\varepsilon_n =\frac{1}{n^2}$ to get
\[
E\bigl[\operatorname{cap} (\Theta_p)\bigr]\leq\sum_{n=p}^{+\infty}
\frac
{2Cn^4}{2^n}.
\]
Therefore,
\[
\lim_{p\rightarrow+\infty}E\bigl[\operatorname{cap} (\Theta_p)\bigr]=0.
\]
For almost all $\omega\in\Omega$, $u^n(\omega)$ is continuous in
$(t,x)$ on $(\Theta_p (w))^c$ and $(u^n(\omega))_n $ converges
uniformly to
$u$ on $(\Theta_p (w))^c$ for all $p$, hence, $u(\omega)$ is continuous
in $(t,x)$ on $(\Theta_p (w))^c$. Then from the definition of
quasi-continuous, we know that $u(\omega)$ admits a quasi-continuous
version since $\operatorname{cap} (\Theta_p )$ tends to $0$ almost surely as $p$
tends to $+\infty$.
\end{pf*}

%s5 #&#
\section{Existence and uniqueness of the solution of the obstacle problem}
%s5.1 #&#
\subsection{Weak solution}
\renewcommand{\theassumption}{(O)}
\begin{assumption}\label{asso}
The obstacle $S$ is assumed to be an adapted
process, quasi-continuous, such that $S_0 \leq\xi$ $P$-almost surely and
controlled by the solution of a SPDE, that is, $\forall t\in[0,T]$,
%
%e15 #&#
\begin{equation}
S_t\leq S'_t,
\end{equation}
where
$S'$ is the solution of the linear SPDE with Dirichlet boundary condition,
%
%e16 #&#
\begin{equation}\label{obstacle}
\cases{ %
\displaystyle dS'_t=LS'_t\,dt+f'_t\,dt+
\sum_{i=1}^d \partial_i
g'_{i,t}\,dt+\sum_{j=1}^{+\infty}h'_{j,t}\,dB^j_t,
\vspace*{2pt}\cr
S'(0)=S'_0,}
\end{equation}
where $S'_0\in L^2 (\Omega\times\cO)$ is
$\mathcal{F}_0$-measurable, and $f'$, $g'$ and $h'$ are adapted
processes, respectively, in $L^2
([0,T]\times\Omega\times\cO;\mathbb{R})$, $L^2
([0,T]\times\Omega\times\cO;\mathbb{R}^d)$ and $L^2
([0,T]\times\Omega\times\cO;\mathbb{R}^{\mathbb{N}^*})$.
\end{assumption}

%
%re5 #&#
\begin{remark}Here again, we know that $S'$ uniquely exists and
satisfies the following estimate:
%
%e17 #&#
\begin{eqnarray}
\label{estimobstacle}&&E\sup_{t\in[0,T]}\bigl\Vert S'_t
\bigr\Vert^2+E\int_0^T\mathcal{E}
\bigl(S'_t\bigr)\,dt
\nonumber
\\[-8pt]
\\[-8pt]
\nonumber
&&\qquad\leq CE \biggl[\bigl\Vert
S'_0\bigr\Vert^2+\int_0^T
\bigl(\bigl\Vert f'_t\bigr\Vert^2+\bigl\Vert
\bigl|g'_t\bigr|\bigr\Vert^2+\bigl\Vert\bigl |h'_t\bigr|
\bigr\Vert^2\bigr)\,dt \biggr].
\end{eqnarray}
Moreover, from Theorem \ref{mainquasicontinuity}, $S'$ admits a
quasi-continuous version.

Let us also remark that even if this assumption seems restrictive since
$S'$ is driven by the same operator and Brownian motions as $u$, it encompasses
a large class of examples.
\end{remark}
We now are able to define rigorously the notion of the solution to the
problem with obstacle we consider.\vadjust{\goodbreak}
%
%de6 #&#
\begin{definition} A pair
$(u,\nu)$ is said to be a solution of the obstacle problem for
(\ref{SPDE}) with Dirichlet boundary condition if:
\begin{longlist}[(1)]
\item[(1)]$u\in\mathcal{H}_T$ and $u(t,x)\geq S(t,x), dP\otimes
dt\otimes
dx$-a.e. and $u_0(x)=\xi, dP\otimes dx$-a.e.;
\item[(2)]$\nu$ is a random regular measure defined on
$[0,T)\times\mathcal{O}$;
\item[(3)] the following relation holds almost surely, for all
$t\in[0,T]$ and $\forall\varphi\in\mathcal{D}$,
%
%e18 #&#
\begin{eqnarray}
\label{solution}&&(u_t,\varphi_t)-(\xi,
\varphi_0)-\int_0^t(u_s,
\partial_s\varphi_s)\,ds+\int_0^t
\mathcal{E}(u_s,\varphi_s)\,ds\nonumber\\
&&\quad{}+\sum
_{i=1}^d\int_0^t
\bigl(g^i_s(u_s,\nabla u_s),
\partial_i\varphi_s\bigr)\,ds
\nonumber
\\[-8pt]
\\[-8pt]
\nonumber
&&\qquad=\int_0^t\bigl(f_s(u_s,
\nabla u_s),\varphi_s\bigr)\,ds+\sum
_{j=1}^{+\infty
}\int_0^t
\bigl(h^j_s(u_s,\nabla u_s),
\varphi_s\bigr)\,dB^j_s\\
&&\qquad\quad{}+\int
_0^t\int_{\mathcal
{O}}
\varphi_s(x)\nu(dx,ds);\nonumber
\end{eqnarray}
\item[(4)]$u$ admits a quasi-continuous version, $\tilde{u}$, and we have
\[
\int_0^T\int_{\mathcal{O}}\bigl(
\tilde{u}(s,x)-{S}(s,x)\bigr)\nu(dx,ds)=0 \qquad \mbox{a.s.}
\]
\end{longlist}
\end{definition}
The main result of this paper is the following:
%
%th4 #&#
\begin{theorem}{\label{maintheo}}
Under Assumptions \textup{\ref{assH}}, \textup{\ref{assI}} and \textup{\ref{asso}}, there exists a
unique weak solution of the obstacle problem for the SPDE
(\ref{SPDE}) associated to $(\xi, f, g, h, S)$.

We denote by $\mathcal{R}(\xi,f,g,h,S)$ the solution of SPDE
(\ref{SPDE}) with obstacle when it exists and is unique.
\end{theorem}

As the proof of this theorem is quite long, we split it in several
steps: first we prove existence and uniqueness in the linear case, then
establish an It\^o formula and finally prove
the theorem thanks to a fixed point argument.
%s5.2 #&#
\subsection{\texorpdfstring{Proof of Theorem \protect\ref{maintheo} in the linear case}
{Proof of Theorem 4 in the linear case}}\label{subsec52}

All along this subsection, we assume that $f$, $g$ and $h$ do not
depend on $u$ and
$\nabla u$, so we consider that $f$, $g$ and $h$ are adapted processes,
respectively, in $L^2 ([0,T]\times\Omega\times\cO;\mathbb{R})$, $L^2
([0,T]\times\Omega\times\cO;\mathbb{R}^d)$ and $L^2 ([0,T]\times
\Omega
\times\cO;\mathbb{R}^{{\mathbb{N}^*}})$.

For $n\in\mathbb{N}^*$, let $u^n$ be the solution of the
following SPDE:
%
%e19 #&#
\begin{equation}
\label{penalization} du_t^n=Lu_t^n\,dt+f_t\,dt+
\sum_{i=1}^d \partial_i
g_{i,t}\,dt+\sum_{j=1}^{+\infty}h_{j,t}\,dB^j_t+n
\bigl(u_t^n-S_t\bigr)^-\,dt
\end{equation}
with initial condition $u^n_0=\xi$ and null Dirichlet boundary
condition. We know from Theorem 8 in \cite{DenisStoica} that this equation
admits a unique solution in $\mathcal{H}_T$ and that the solution
admits $L^2(\mathcal{O})$-continuous trajectories.

%le4 #&#
\begin{lemma}\label{lemmaestim}
For all $n\in\bbN^*$, $u^n$ satisfies the following
estimate:
\begin{eqnarray*}
E\sup_{t\in[0,T]}\bigl\Vert u_t^n\bigr\Vert^2+E\int
_0^T\mathcal{E}\bigl(u_t^n
\bigr)\,dt+E\int_0^Tn\bigl\Vert \bigl(u_t^n-S_t
\bigr)^-\bigr\Vert^2\,dt\leq C,
\end{eqnarray*}
where $C$ is a constant depending only on the structure constants of
the SPDE.
\end{lemma}
\begin{pf} From
(\ref{penalization}) and (\ref{obstacle}), we know that $u^n-S'$
satisfies the following equation:
\[
d\bigl(u_t^n-S'_t\bigr)=L
\bigl(u_t^n-S'_t\bigr)\,dt+
\tilde{f}_t\,dt+\sum_{i=1}^d
\partial_i\tilde {g}^i_t\,dt+\sum
_{j=1}^{+\infty}\tilde{h}^j_t\,dB^j_t+n
\bigl(u_t^n-S_t\bigr)^-\,dt,
\]
where $\tilde{f}=f-f'$, $\tilde{g}=g-g'$ and $\tilde{h}=h-h'$.
Applying It\^o's formula to $(u^n-S')^2$, we
have
\begin{eqnarray*}
&&\bigl\Vert u_t^n-S'_t
\bigr\Vert^2+2\int_0^t\mathcal{E}
\bigl(u_s^n-S'_s\bigr)\,ds\\
&&\qquad=2
\int_0^t\bigl(\bigl(u_s^n-S'_s
\bigr),\tilde{f}_s\bigr)\,ds+2\sum_{j=1}^{+\infty}
\int_0^t\bigl(\bigl(u_s^n-S'_s
\bigr),\tilde{h}^j_s\bigr)\,dB^j_s
\\
&&\qquad\quad{}-2\sum_{i=1}^d\int_0^t
\bigl(\partial_i\bigl(u_s^n-
\tilde{S}_s\bigr),\tilde {g}^i_s\bigr)\,ds+2
\int_0^t\int_\mathcal {O}
\bigl(u_s^n-S'_s\bigr)n
\bigl(u_s^n-S_s\bigr)^-\,ds\\
&&\qquad\quad{}+\int
_0^t\bigl\Vert \bigl|\tilde{h}_s\bigr|
\bigr\Vert^2\,ds,\qquad \mbox{a.s.}
\end{eqnarray*}
We remark first
\begin{eqnarray*}
&&\int_0^t\int_\mathcal {O}
\bigl(u_s^n-S'_s\bigr)n
\bigl(u_s^n-S_s\bigr)^-\,ds\\
&&\qquad=\int
_0^t\int_\mathcal {O}
\bigl(u_s^n-S_s+S_s-S'_s
\bigr)n\bigl(u_s^n-S_s\bigr)^-\,ds
\\
&&\qquad = -\int_0^t\int_\mathcal{O}n
\bigl(\bigl(u_s^n-S_s\bigr)^-
\bigr)^2\,ds+\int_0^t\int
_\mathcal {O}\bigl(S_s-S'_s
\bigr)n\bigl(u_s^n-S_s\bigr)^-\,dx\,ds;
\end{eqnarray*}
the last term in the right member is nonpositive because $S_t\leq
S'_t$, thus,
\begin{eqnarray*}
&&\bigl\Vert u_t^n-S'_t
\bigr\Vert^2+2\int_0^t\mathcal{E}
\bigl(u_s^n-S'_s\bigr)\,ds+2
\int_0^tn\bigl\Vert \bigl(u_s^n-S
\bigr)^-\bigr\Vert^2\,ds\\
&&\qquad\leq2\int_0^t
\bigl(u_s^n-S'_s,\tilde
{f}_s\bigr)\,ds
-2\sum_{i=1}^d\int_0^t
\bigl(\partial_i\bigl(u_s^n-S'_s
\bigr),\tilde{g}^i_s\bigr)\,ds\\
&&\qquad\quad{}+ 2\sum
_{j=1}^{+\infty}\int_0^t
\bigl(u_s^n-S'_s,
\tilde{h}^j_s\bigr)\,dB^j_s +\int
_0^t\bigl\Vert |\tilde{h}_s|
\bigr\Vert^2\,ds\qquad \mbox{a.s.}
\end{eqnarray*}
Then
using Cauchy--Schwarz's inequality, we have $\forall t\in[0,T]$,
\begin{eqnarray*}
2\biggl|\int_0^t\bigl(u_s^n-S'_s,
\tilde{f}_s\bigr)\,ds\biggr|\leq\varepsilon \int_0^T
\bigl\Vert u_s^n-S'_s
\bigr\Vert^2\,ds+\frac{1}{\varepsilon}\int_0^T
\Vert \tilde{f}_s\Vert^2\,ds
\end{eqnarray*}
and
\begin{eqnarray*}
&&2\Biggl|\sum_{i=1}^d\int_0^t
\bigl(\partial_i\bigl(u_s^n-S'_s
\bigr),\tilde {g}^i_s\bigr)\,ds\Biggr|\\
&&\qquad\leq\varepsilon\int
_0^T\bigl\Vert \nabla \bigl(u_s^n-S'_s
\bigr)\bigr\Vert^2\,ds+\frac{1}{\varepsilon}\int_0^T
\bigl\Vert |\tilde{g}|\bigr\Vert^2\,ds.
\end{eqnarray*}
Moreover, thanks to Burkholder--Davies--Gundy's inequality, we get
\begin{eqnarray*}
&& E\sup_{t\in[0,T]}\Biggl|\sum_{j=1}^{+\infty}\int
_0^t\bigl(u_s^n-S'_s,
\tilde{h}^j_s\bigr)\,dB^j_s\Biggr|\\
&&\qquad
\leq c_1E\Biggl[\int_0^T\sum
_{j=1}^{+\infty}\bigl(u_s^n-S'_s,
\tilde{h}^j_s\bigr)^2\,ds\Biggr]^{1/2}
\\
&&\qquad\leq c_1E\Biggl[\int_0^T\sum
_{j=1}^{+\infty}\sup_{s\in[0,T]}\bigl\Vert
u_s^n-S'_s
\bigr\Vert^2\bigl\Vert \tilde{h}^j_s
\bigr\Vert^2\,ds\Biggr]^{1/2}
\\
&&\qquad\leq c_1E\biggl[\sup_{s\in[0,T]}\bigl\Vert u_s^n-S'_s
\bigr\Vert \biggl(\int_0^T\bigl\Vert |
\tilde{h}_s|\bigr\Vert^2\,ds\biggr)^{1/2}\biggr]
\\
&&\qquad \leq\varepsilon E\sup_{s\in[0,T]}\bigl\Vert u_s^n-S'_s
\bigr\Vert^2+\frac{c_1}{4\varepsilon}E\int_0^T
\bigl\Vert |\tilde {h}_s|\bigr\Vert^2\,ds.
\end{eqnarray*}
Then
using the strict ellipticity assumption and the inequalities above,
we get
\begin{eqnarray*}
&&\bigl(1-2\varepsilon(T+1)\bigr)E\sup_{t\in[0,T]}\bigl\Vert u_t^n-S'_t
\bigr\Vert^2+(2\lambda-\varepsilon)E\int_0^T
\mathcal {E}\bigl(u_s^n-S'_s
\bigr)\,ds\\
&&\quad{}+2E\int_0^Tn\bigl\Vert \bigl(
u_s^n-S_s\bigr)^-\bigr\Vert^2\,ds
\\
&&\qquad\leq C\biggl(E\Vert \xi\Vert^2+\frac{2}{\varepsilon}E\int
_0^T\Vert \tilde{f}_s
\Vert^2+\frac{2}{\varepsilon}\bigl\Vert |\tilde{g}_s|
\bigr\Vert^2+\biggl(\frac{c_1}{2\varepsilon}+1\biggr)\bigl\Vert |\tilde{h}_s|
\bigr\Vert^2\,ds\biggr).
\end{eqnarray*}
We take $\varepsilon$
small enough such that $(1-2\varepsilon(T+1))>0$; this yields
$(2\lambda-\varepsilon)>0,$
\begin{eqnarray*}
E\sup_{t\in[0,T]}\bigl\Vert u_t^n-S'_t
\bigr\Vert^2+E\int_0^T\mathcal{E}
\bigl(u_t^n-S'_t\bigr)\,dt+E\int
_0^Tn\bigl\Vert \bigl(u_t^n-S_t
\bigr)^-\bigr\Vert^2\,dt\leq C.
\end{eqnarray*}
Then with (\ref{estimobstacle}), we obtain the desired estimate.
\end{pf}
We now introduce $z$, the solution of the corresponding SPDE without obstacle:
\begin{eqnarray*}
dz_t+ Az_t \,dt=f_t\,dt+\sum
_{i=1}^d\partial_{i} g_{i,t}\,dt+
\sum_{j=1}^{+\infty}h_{j,t}\,dB_t^j,
\end{eqnarray*}
starting from $z_0 =\xi$, with null Dirichlet condition on the
boundary. As a consequence of Theorem \ref{mainquasicontinuity}, we can
take for $z$ a quasi-continuous version.

For each $n\in\mathbb{N}^*$, we put $v^n =u^n-z$. Clearly, $v^n$ satisfies
\[
dv_t^n+Av_t^n\,dt=n
\bigl(v_t^n-(S_t-z_t)\bigr)^-\,dt=n
\bigl(u^n_t -S_t \bigr)^- \,dt.
\]
Since $S-z$ is quasi-continuous almost-surely, by the results
established by Mignot and Puel in \cite{MignotPuel}, we know that
$P$-almost surely, the sequence $(v^n )_n$ is increasing and
converges in $L^2 ([0,T]\times\cO)$ $P$-almost surely to $v$ and
that the sequence of random measures $\nu^{v^n}=n(u^n_t -S_t )^-\,
dt\,dx$ converges vaguely to a measure associated to $v$: $\nu=\nu^v$.
As a consequence of the previous lemma, $(u^n)_n$ and $(v^n )_n$ are
bounded sequences in $L^2(\Omega\times[0,T];H_0^1(\mathcal{O})),$
which is a Hilbert space [equipped the norm $(E\int_0^T\Vert
u_t\Vert^2_{H_0^1(\mathcal{O})}\,dt)^{1/2}$]. By a double extraction
argument, we can construct subsequences $(u^{n_k})_k$ and
$(v^{n_k})_k$ such that the first one converges weakly in
$L^2(\Omega\times[0,T];H_0^1(\mathcal{O}))$ to an element that we
denote $u$ and the second one to an element which necessarily is equal
to $v$ since $(v^n )_n$ is increasing. Moreover, we can construct
sequences $(\hat{u}^n)_n$ and $(\hat{v}^n)_n$ of convex combinations of
elements
of the form
\[
\hat{u}^n=\sum_{k=1}^{N_n}
\alpha_k^nu^{n_k} \quad\mbox{and}\quad\hat
{v}^n=\sum_{k=1}^{N_n}
\alpha_k^n v^{n_k}
\]
converging strongly to $u$ an $v$, respectively, in
$L^2(\Omega\times[0,T];H_0^1(\mathcal{O}))$.

From the fact that $u^n$ is the weak solution of
(\ref{penalization}), we get
%
%e20 #&#
\begin{eqnarray}
\label{weaksol}&&\bigl(u^n_t,\varphi_t
\bigr)-(\xi,\varphi_0)-\int_0^t
\bigl(u^n_s,\partial_s\varphi_s
\bigr)\,ds+\int_0^t\mathcal{E}
\bigl(u^n_s,\varphi_s\bigr)\,ds\nonumber\\
&&\quad{}+\sum
_{i=1}^d\int_0^t
\bigl(g^i_s,\partial_i\varphi_s
\bigr)\,ds
\nonumber
\\[-8pt]
\\[-10pt]
\nonumber
&&\qquad =\int_0^t(f_s,
\varphi_s)\,ds+\sum_{j=1}^{+\infty}\int
_0^t\bigl(h^j_s,
\varphi_s\bigr)\,dB^j_s\\[-2pt]
&&\qquad\quad{}+\int
_0^t\int_{\mathcal{O}}
\varphi_s(x)n\bigl(u^n_s-S_s
\bigr)^-\,dx\,ds\qquad \mbox{a.s.}\nonumber
\end{eqnarray}
Hence,
%
%e21 #&#
\begin{eqnarray}
\label{weaksol2}&&\bigl(\hat{u}^n_t,
\varphi_t\bigr)-(\xi,\varphi_0)-\int_0^t
\bigl(\hat{u}^n_s,\partial_s
\varphi_s\bigr)\,ds+\int_0^t\mathcal
{E}\bigl(\hat {u}^n_s,\varphi_s\bigr)\,ds+
\sum_{i=1}^d\int_0^t
\bigl(g^i_s,\partial_i\varphi_s
\bigr)\,ds
\nonumber
\\[-2pt]
&&\qquad=\int_0^t(f_s,
\varphi_s)\,ds+\sum_{j=1}^{+\infty}\int
_0^t\bigl(h^j_s,
\varphi_s\bigr)\,dB^j_s\\[-2pt]
&&\qquad\quad{}+\int
_0^t\int_{\mathcal{O}}
\varphi_s(x) \Biggl( \sum_{k=1}^{N_n}
n_k\bigl(u^{n_k}_s-S_s\bigr)^-
\Biggr)\,dx\,ds\qquad \mbox{a.s.}\nonumber
\end{eqnarray}

We have
\begin{eqnarray*}
\int_0^t\int_{\mathcal{O}}
\varphi_s(x) \Biggl( \sum_{k=1}^{N_n}
n_k\bigl(u^{n_k}_s-S_s\bigr)^-
\Biggr)\,dx\,ds&=&\int_0^T-\biggl(\frac{\partial\varphi_t}{\partial
t},
\hat{v}^n_t\biggr)\,dt+\int_0^T
\mathcal{E}\bigl(\varphi_t,\hat{v}^n_t
\bigr)\,dt
\end{eqnarray*}
so that we have almost surely, at least for a subsequence,
\begin{eqnarray*}
&&\lim_{n\rightarrow+\infty}\int_0^t\int
_{\mathcal
{O}}\varphi_s(x) \Biggl( \sum
_{k=1}^{N_n} n_k\bigl(u^{n_k}_s-S_s
\bigr)^- \Biggr)\,dx\,ds\\[-2pt]
&&\qquad=\int_0^T-\biggl(
\frac{\partial\varphi_t}{\partial
t},{v}_t\biggr)\,dt+\int_0^T
\mathcal{E}(\varphi_t,{v}_t)\,dt
\\[-2pt]
&&\qquad=\int_0^T\int_{\cO
}
\varphi_t (x)\nu(dx, dt).
\end{eqnarray*}
As $(\hat{u}^n)_n$ converges to $u$
in $L^2(\Omega\times[0,T];H_0^1(\mathcal{O}))$, by making $n$ tend to
$+\infty$ in~\eqref{weaksol2}, we obtain
\begin{eqnarray*}
&&(u_t,\varphi_t)-(\xi,\varphi_0)-\int
_0^t(u_s,\partial_s
\varphi_s)\,ds+\int_0^t
\mathcal{E}(u_s,\varphi_s)\,ds+\sum
_{i=1}^d\int_0^t
\bigl(g^i_s,\partial_i\varphi_s
\bigr)\,ds
\\[-2pt]
&&\qquad=\int_0^t(f_s,
\varphi_s)\,ds+\sum_{j=1}^{+\infty}\int
_0^t\bigl(h^j_s,
\varphi_s\bigr)\,dB^j_s+\int
_0^t\int_{\mathcal{O}}
\varphi_s(x)\nu(dx,ds)\qquad \mbox{a.s.}
\end{eqnarray*}
In the next subsection we will show that $u$ satisfies an It\^o formula.
As a consequence by applying it to $u_t^2$, using standard arguments,
we get that $u\in\mathcal{H}_T$ so for almost all
$\omega\in\Omega$, $u(\omega)\in\mathcal{K}$. And from Theorem 9 in
\cite{DenisStoica}, we know that for almost all $\omega\in\Omega$,
$z(\omega)\in\mathcal{K}$. Therefore,\vadjust{\goodbreak} for almost all $\omega\in
\Omega$,
$v(\omega)=u(w)-z(w)\in\mathcal{K}$. Hence, $\nu=\partial_tv+Av$
is a regular
measure by definition. Moreover, by \cite{Pierre,PIERRER} we know
that $v$ admits a quasi-continuous version $\tilde{v}$ which
satisfies the minimality condition
%
%e22 #&#
\begin{equation}
\label{minimal}\int\int(\tilde{v}-S+\tilde{z})\nu(dx\,dt)=0.
\end{equation}
$z$ is quasi-continuous version, hence, $\tilde{u}=z+\tilde{v}$ is a
quasi-continuous version of $u$ and we can write \eqref{minimal} as
\[
\int\int(\tilde{u}-S)\nu(dx\,dt)=0.
\]
The fact that $u\geq S$ comes from the fact that $v\geq u-z$, so at
this stage we have proved that $(u,\nu)$ is a solution to the obstacle
problem we consider.

Uniqueness comes from the fact that both $z$ and $v$ are unique, which
ends the proof of Theorem \ref{maintheo}.

%s5.3 #&#
\subsection{It\^o's formula}
The following It\^o formula for the solution of the obstacle problem is
fundamental to get all the results in the nonlinear case.
Let us also remark that any solution of the nonlinear equation \eqref
{SPDE} may be viewed as the solution of a linear one so that it also
satisfies the It\^o formula.

%th5 #&#
\begin{theorem}\label{Itoformula}
Under assumptions of the previous Section \ref{subsec52}, let $u$
be the solution of SPDE (\ref{SPDE}) with obstacle and
$\Phi\dvtx \mathbb{R}^+\times\mathbb{R}\rightarrow\mathbb{R}$ be a
function of class $\mathcal{C}^{1,2}$. We denote by $\Phi'$ and
$\Phi''$ the derivatives of $\Phi$ with respect to the space
variables and by $\frac{\partial\Phi}{\partial t}$ the partial
derivative with respect to time. We assume that these derivatives
are bounded and $\Phi'(t,0)=0$ for all $t\geq0$. Then $P$-a.s. for
all $t\in[0,T]$,
\begin{eqnarray*}
&&\int_\mathcal{O}\Phi\bigl(t,u_t(x)\bigr)\,dx+\int
_0^t\mathcal{E}\bigl(\Phi'(s,u_s),u_s
\bigr)\,ds\\[-2pt]
&&\qquad=\int_\mathcal{O}\Phi\bigl(0,\xi(x)\bigr)\,dx+\int
_0^t\int_\mathcal {O}
\frac{\partial\Phi}{\partial
s}\bigl(s,u_s(x)\bigr)\,dx\,ds
\\[-2pt]
&&\qquad\quad{}+\int_0^t\bigl(\Phi'(s,u_s),f_s
\bigr)\,ds -\sum_{i=1}^d\int
_0^t\int_\mathcal{O}
\Phi''\bigl(s,u_s(x)\bigr)
\partial_iu_s(x)g_i(x)\,dx\,ds\\[-2pt]
&&\qquad\quad{}+\sum
_{j=1}^{+\infty}\int_0^t
\bigl(\Phi'(s,u_s),h_j\bigr)\,dB_s^j
\\[-2pt]
&&\qquad\quad{} +\frac{1}{2}\sum_{j=1}^{+\infty}\int
_0^t\int_\mathcal{O}
\Phi''\bigl(s,u_s(x)\bigr)
\bigl(h_{j,s}(x)\bigr)^2\,dx\,ds\\[-2pt]
&&\qquad\quad{}+\int_0^t
\int_\mathcal{O}\Phi'\bigl(s,\tilde
{u}_s(x)\bigr)\nu(dx\,ds).
\end{eqnarray*}
\end{theorem}
\begin{pf} We keep the same notation as in the previous subsection
and so consider the sequence $(u^n )_n$ approximating $u$ and also
$(\hat{u}^n)_n$ the sequence of convex combinations $\hat{u}^n=\sum_{k=1}^{N_n}\alpha_k^nu^{n_k}$
converging strongly to $u$ in
$L^2(\Omega\times[0,T];H_0^1(\mathcal{O}))$.

Moreover, by standard arguments such as the Banach--Saks theorem, since
$(u^n)_n$ is nondecreasing, we can choose the convex combinations such that
$(\hat{u}^n )_n$ is also a nondecreasing sequence.
We start by a key lemma:
%
%le5 #&#
\begin{lemma}\label{crucial}Let $ t\in[0,T]$, then
\[
\lim_{n\rightarrow+\infty}E\int_0^t\int
_\mathcal{O}\bigl(\hat {u}^n_s-S_s
\bigr)^-\sum_{k=1}^{N_n}\alpha^n_k
n_k \bigl(u^{n_k}_s-S_s
\bigr)^-\,dx\,ds =0.
\]
\end{lemma}
\begin{pf}
We write as above $u^n =v^n +z$ and we denote $\hat{\nu}^n=\break \sum_{k=1}^{N_n}\alpha^n_kn_k(u^{n_k}_s-S_s)^-$ so that
\[
\int_0^t\int_\mathcal{O}
\bigl(\hat{u}^n_s-S_s\bigr)\hat{
\nu}^n(dx\,ds)=\int_0^t\int
_\mathcal{O}\hat{v}^n_s\hat{
\nu}^n(dx\,ds) + \int_0^t\int
_\mathcal {O}(z_s-S_s)\hat{
\nu}^n(dx\,ds).
\]
From Lemma \ref{convergemeas}, we know that
\[
\int_0^t\int_\mathcal{O}(z_s-S_s)
\hat{\nu}^n(dx\,ds)\rightarrow\int_0^t
\int_\mathcal{O}(z_s-S_s)\nu(dx\,ds).
\]
Moreover, by Lemma II.6 in \cite{Pierre}, we have for all $n$
\[
\frac{1}{2}\bigl\Vert \hat{v}^n_T\bigr\Vert^2+
\int_0^T\mathcal {E}\bigl(\hat
{v}^n_s\bigr)\,ds=\int_0^T
\int_\mathcal{O}\hat{v}^n_s\hat{
\nu}^n(dx\,ds)
\]
and
\[
\frac{1}{2}\Vert v_{T}\Vert^2+\int
_0^T\mathcal {E}(v_s)\,ds=\int
_0^T\int_\mathcal{O}
\tilde{v}_s\nu(dx\,ds).
\]
As $(\hat{v}^n )_n$ tends to $v$ in $L^2 ([0,T], H^1_0 (\cO))$,
\[
\lim_{n\rightarrow+\infty}\int_0^T\mathcal{E}\bigl(
\hat{v}^n_s\bigr)\,ds =\int_0^T
\mathcal{E}({v}_s)\,ds.
\]
Let us prove that $(\Vert \hat{v}^n_{T}\Vert )_n$ tends to
$\Vert  v_{T}\Vert $.

Since $(\hat{v}^n_{T})_n$ is nondecreasing and bounded in $L^2 (\cO)$,
it converges in $L^2 (\cO)$ to $ m=\sup_n \hat{v}^n_{T}$.
Let $\rho\in H^1_0 (\cO)$, then the map defined by $\varphi
(t,x)=\rho(x)$ belongs to~$\mathcal{W}$, hence, as a consequence of
Proposition \ref{approximation},
\[
\int_{[0,T[\times\cO}\rho \,d\hat{\nu}^n = \int
_0^T \mathcal {E}\bigl(\rho,
\hat{v}_s^n\bigr) \,ds +\bigl(\rho, \hat{v}^n_T
\bigr)
\]
and
\[
\int_{[0,T[\times\cO}\tilde{\rho} \,d{\nu} = \int_0^T
\mathcal {E}(\rho, {v_s}) \,ds +(\rho, {v}_T);
\]
making $n$ tend to $+\infty$ and using one more time Lemma \ref
{convergemeas}, we get
\[
\lim_{n\rightarrow+\infty}\bigl(\rho, \hat{v}^n_T\bigr)=(\rho,
m)=(\rho, {v}_T).
\]
Since $\rho$ is arbitrary, we have $v_T =m$ and so $\lim_{n\rightarrow
+\infty}\Vert \hat{v}^n_{T}\Vert =\Vert  v_{T}\Vert $ and
this yields
\[
\lim_{n\rightarrow+\infty}\int_0^T\int
_\mathcal{O}\hat {v}^n_s\hat{\nu
}^n(dx\,ds)=\int_0^T\int
_\mathcal{O}\tilde{v}_s\nu(dx\,ds)=\int
_0^T\int_\mathcal{O}(S_s
-z_s){\nu} (dx\,ds).
\]
This proves that
\[
\lim_{n\rightarrow+\infty}\int_0^t\int
_\mathcal{O}\bigl(\hat {u}^n_s-S_s
\bigr)\hat {\nu}^n(dx\,ds)=0.
\]
We conclude by remarking that
\begin{eqnarray*}
\lim_{n\rightarrow+\infty}\int_0^t\int
_\mathcal{O}\bigl(\hat {u}^n_s-S_s
\bigr)^+\hat{\nu}^n(dx\,ds)&\leq&\lim_{n\rightarrow+\infty}\int
_0^t\int_\mathcal{O}({u}_s-S_s)
\hat{\nu}^n(dx\,ds)\\
&=&\int_0^t\int
_\mathcal {O}(\tilde{u}_s-S_s){
\nu}(dx\,ds)=0.\hspace*{24pt}\qquad\qed
\end{eqnarray*}
\noqed\end{pf}
We now end the proof of Theorem \ref{Itoformula}.
We
consider the penalized solution $(u^n)_n$, and we know that its convex
combination $(\hat{u}^n)_n$ converges strongly to $u$ in
$L^2(\Omega\times[0,T];H_0^1(\mathcal{O}))$. And $\hat{u}^n$
satisfies the following SPDE:
\begin{eqnarray*}
d\hat{u}_t^n+A\hat{u}_t^n\,dt=f_t\,dt+
\sum_{i=1}^d\partial_ig^i_t\,dt+
\sum_{j=1}^{+\infty}h^j_t\,dB^j_t+
\sum_{k=1}^{N_n}\alpha^n_k
n_k \bigl(u^{n_k}_s-S_s
\bigr)^-\,dt.
\end{eqnarray*}
From the It\^o formula for the solution of SPDE without obstacle
(see Lemma 7 in~\cite{DMS05}), we have, almost surely, for all
$t\in[0,T]$,
\begin{eqnarray*}
&&\int_\mathcal{O}\Phi\bigl(t,\hat{u}_t^n(x)
\bigr)\,dx+\int_0^t\mathcal {E}\bigl(
\Phi'\bigl(s,\hat{u}_s^n\bigr),
\hat{u}_s^n\bigr)\,ds\\[-2pt]
&&\qquad=\int_\mathcal{O}
\Phi\bigl(0,\xi (x)\bigr)\,dx+\int_0^t\int
_\mathcal{O}\frac{\partial\Phi}{\partial
s}\bigl(s,\hat{u}_s^n
\bigr)\,dx\,ds
\\[-2pt]
&&\qquad\quad{}+\int_0^t\bigl(\Phi'\bigl(s,
\hat{u}_s^n\bigr),f_s\bigr)\,ds\\[-2pt]
&&\qquad\quad{} -\sum
_{i=1}^d\int_0^t
\int_\mathcal{O}\Phi''\bigl(s,\hat
{u}_s^n(x)\bigr)\partial_i
\hat{u}_s^n(x)g_i(x)\,dx\,ds\\[-2pt]
&&\qquad\quad{}+\sum
_{j=1}^{+\infty}\int_0^t
\bigl(\Phi'\bigl(s,\hat {u}_s^n
\bigr),h_j\bigr)\,dB_s^j
\\[-2pt]
&&\qquad\quad{} +\frac{1}{2}\sum_{j=1}^{+\infty}\int
_0^t\int_\mathcal{O}
\Phi''\bigl(s,\hat {u}_s^n(x)
\bigr) \bigl(h_j(x)\bigr)^2\,dx\,ds\\[-2pt]
&&\qquad\quad{}+\int_0^t
\int_\mathcal{O}\Phi'\bigl(s,\hat
{u}^n_s\bigr)\sum_{k=1}^{N_n}
\alpha^n_k n_k \bigl(u^{n_k}_s-S_s
\bigr)^-\,dx\,ds.
\end{eqnarray*}
Because of the strong convergence of $(\hat{u}^n)_n$, the convergence of
all the terms except the last one are clear. To obtain the
convergence of the last term, we do as follows:
\begin{eqnarray*}
&&\int_0^t\int_\mathcal{O}
\Phi'\bigl(s,\hat{u}^n_s\bigr)\sum
_{k=1}^{N_n}\alpha^n_k
n_k \bigl(u^{n_k}_s-S_s
\bigr)^-\,dx\,ds
\\[-2pt]
&&\qquad=\int_0^t\int_\mathcal
{O}\bigl(\Phi'\bigl(s,\hat{u}^n_s\bigr)-
\Phi'(s,S_s)\bigr)\sum_{k=1}^{N_n}
\alpha^n_k n_k \bigl(u^{n_k}_s-S_s
\bigr)^-\,dx\,ds
\\[-2pt]
&&\qquad\quad{}+\int_0^t\int_\mathcal{O}
\Phi'(s,S_s)\sum_{k=1}^{N_n}
\alpha^n_k n_k \bigl(u^{n_k}_s-S_s
\bigr)^-\,dx\,ds.
\end{eqnarray*}
For the first term in the right member, we have
\begin{eqnarray*}
&&\Biggl|\int_0^t\int_\mathcal{O}
\bigl(\Phi'\bigl(s,\hat {u}^n_s\bigr)-
\Phi'(s,S_s)\bigr)\sum_{k=1}^{N_n}
\alpha^n_k n_k \bigl(u^{n_k}_s-S_s
\bigr)^-\,dx\,ds\Biggr|
\\[-2pt]
&&\qquad\leq C\int_0^t\int_\mathcal{O}\bigl|
\hat{u}^n_s -S_s\bigr|\cdot\sum
_{k=1}^{N_n}\alpha^n_k
n_k \bigl(u^{n_k}_s-S_s
\bigr)^-\,dx\,ds
\\[-2pt]
&&\qquad=C\int_0^t\int_\mathcal{O}
\bigl(\bigl(\hat{u}^n_s-S_s\bigr)^++\bigl(
\hat{u}_s^n-S_s\bigr)^-\bigr)\sum
_{k=1}^{N_n}\alpha^n_k
n_k \bigl(u^{n_k}_s-S_s
\bigr)^-\,dx\,ds
\\[-2pt]
&&\qquad=C\int_0^t\int_\mathcal{O}
\bigl(\hat{u}^n_s-S_s\bigr)^+\sum
_{k=1}^{N_n}\alpha^n_k
n_k \bigl(u^{n_k}_s-S_s
\bigr)^-\,dx\,ds\\[-2pt]
&&\qquad\quad{}+C\int_0^t\int_\mathcal{O}
\bigl(\hat {u}^n_s-S_s\bigr)^-\sum
_{k=1}^{N_n}\alpha^n_k
n_k \bigl(u^{n_k}_s-S_s
\bigr)^-\,dx\,ds.
\end{eqnarray*}
We have the following inequality because $(\hat{u}^n)_n$
converges to $u$ increasingly:
\begin{eqnarray*}
&&\int_0^t\int_\mathcal{O}\bigl(
\hat{u}^n_s-S_s\bigr)^- \sum
_{k=1}^{N_n}\alpha^n_k
n_k \bigl(u^{n_k}_s-S_s
\bigr)^-\,dx\,ds\\
&&\qquad\leq\int_0^t\int
_\mathcal{O}(u_s-S_s)^+\sum
_{k=1}^{N_n}\alpha^n_k
n_k \bigl(u^{n_k}_s-S_s
\bigr)^-\,dx\,ds
\\
&&\qquad= \int_0^t\int_\mathcal{O}(u_s-S_s)
\sum_{k=1}^{N_n}\alpha^n_k
n_k \bigl(u^{n_k}_s-S_s
\bigr)^-\,dx\,ds.
\end{eqnarray*}
With Lemma \ref{convergemeas}, we know that
\[
\lim_{n\rightarrow\infty}\int_0^t\int
_\mathcal{O}(u_s-S_s)\sum
_{k=1}^{N_n}\alpha^n_k
n_k \bigl(u^{n_k}_s-S_s
\bigr)^-\,dx\,ds\rightarrow\int_0^t\int
_\mathcal{O}(\tilde{u}_s-\tilde{S}_s)
\nu(dx\,ds)=0.
\]
And from Lemma \ref{crucial}, we
have
\[
\int_0^t\int_\mathcal{O}
\bigl(\hat{u}^n_s-S_s\bigr)^-\sum
_{k=1}^{N_n}\alpha^n_k
n_k \bigl(u^{n_k}_s-S_s
\bigr)^-\,dx\,ds\rightarrow0.
\]
Therefore,
\begin{eqnarray*}
\int_0^t\int_\mathcal{O}\bigl(
\Phi'\bigl(s,\hat {u}_s^n\bigr)-
\Phi'(s,S_s)\bigr) \sum_{k=1}^{N_n}
\alpha^n_k n_k \bigl(u^{n_k}_s-S_s
\bigr)^-\,dx\,ds\rightarrow0.
\end{eqnarray*}
Moreover, with Lemma \ref{convergemeas}, we have
\begin{eqnarray*}
\int_0^t\int_\mathcal{O}
\Phi'(s,S_s)\sum_{k=1}^{N_n}
\alpha^n_k n_k \bigl(u^{n_k}_s-S_s
\bigr)^-\,dx\,ds\rightarrow\int_0^t\int
_\mathcal{O}\Phi'(s,S_s)\nu(dx\,ds)
\end{eqnarray*}
and
\begin{eqnarray*}
&&\biggl|\int_0^t\int_\mathcal{O}
\Phi'(s,u_s)\nu (dx\,ds)-\int_0^t
\int_\mathcal{O}\Phi'(s,S_s)
\nu(dx\,ds)\biggl|\\
&&\qquad\leq C\int_0^t\int
_\mathcal{O}|\tilde{u}_s-S_s|\nu(dx\,ds)
\\
&&\qquad=C\int_0^t\int_\mathcal{O}(
\tilde{u}_s-S_s)\nu(dx\,ds)=0.
\end{eqnarray*}
Therefore, taking the limit, we get the desired It\^o formula.
\end{pf}
%
%s5.4 #&#
\subsection{It\^o's formula for the difference of the solutions of
two OSPDEs} We still consider $(u,\nu)$ the solution of the linear
equation as in Section \ref{subsec52},
\[
\cases{ %
\displaystyle du_t+Au_t\,dt=f_t\,dt+
\sum_{i=1}^d\partial_ig^i_t\,dt+
\sum_{j=1}^{+\infty
}h^j_t\,dB^j_t
+\nu(dt,x),
\vspace*{2pt}\cr
u\geq S,}
\]
and consider another linear equation with adapted coefficients $\bar
{f}$, $\bar{g}$, $\bar{h}$, respectively, in $L^2 ([0,T]\times\Omega
\times
\cO;\mathbb{R})$, $L^2 ([0,T]\times\Omega\times\cO;\mathbb
{R}^d)$ and
$L^2 ([0,T]\times\Omega\times\cO;\mathbb{R}^{{\mathbb{N}^*}})$ and
obstacle $\bar{S}$ which satisfies the same hypotheses \ref{asso} as
$S$, that is, $\bar{S}_0 \leq\xi$ and~$\bar{S}$ is dominated by the
solution of an SPDE (not necessarily the same as $S$). We denote by
$(y, \bar{\nu})$ the unique solution to the associated SPDE with
obstacle with initial condition $y_0 =u_0 =\xi$:
\[
\cases{ %
 \displaystyle dy_t+Ay_t\,dt=
\bar{f}_t\,dt+\sum_{i=1}^d
\partial_i\bar{g}^i_t\,dt+\sum
_{j=1}^{+\infty}\bar{h}^j_t\,dB^j_t
+\bar{\nu}(dt,x),
\vspace*{2pt}\cr
y\geq \bar{S},}
\]

%th6 #&#
\begin{theorem}\label{Itodifference}Let $\Phi$ as in Theorem \ref
{Itoformula}, then the difference of the two solutions satisfy the
following It\^o
formula for all $t\in[0,T]$:
%
%e23 #&#
\begin{eqnarray}
\label{itodifference}
&&\int_\mathcal{O}\Phi
\bigl(t,u_t(x)-y_t(x)\bigr)\,dx+\int_0^t
\mathcal{E}\bigl(\Phi'(s,u_s-y_s),u_s-y_s
\bigr)\,ds\nonumber\\
&&\qquad=\int_0^t\bigl(\Phi'(s,u_s-y_s),f_s-
\bar{f}_s\bigr)\,ds\nonumber\\
&&\qquad\quad{} -\sum_{i=1}^d\int_0^t
\int_\mathcal{O}\Phi''(s,u_s-y_s)
\partial_i(u_s-y_s) \bigl(g^i_s-
\bar{g}^i_s\bigr)\,dx\,ds\\
&&\qquad\quad{}+\sum_{j=1}^{+\infty}
\int_0^t\bigl(\Phi'(s,u_s-y_s),h^j_s-
\bar{h}^j_s\bigr)\,dB^j_s
\nonumber\\
&&\qquad\quad{}+\frac
{1}{2}\sum_{j=1}^{+\infty}\int
_0^t\int_\mathcal{O}
\Phi''(s,u_s-y_s)
\bigl(h^j_s-\bar{h}^j_s
\bigr)^2\,dx\,ds\nonumber\\
&&\qquad\quad{}+\int_0^t\int
_\mathcal{O}\frac{\partial
\Phi
}{\partial
s}(s,u_s-y_s)\,dx\,ds
\nonumber
\nonumber\\
&&\qquad\quad{}+\int_0^t\int_\mathcal{O}
\Phi'(s,\tilde {u}_s-\tilde{y}_s) (\nu-
\bar{\nu}) (dx, ds)\qquad \mbox{a.s.}\nonumber
\end{eqnarray}
\end{theorem}
\begin{pf} We
begin with the penalized solutions. The corresponding penalization
equations are
\[
du^n_t+Au^n_t\,dt=f_t\,dt+
\sum_{i=1}^d\partial_i
g^i_t\,dt+\sum_{j=1}^{+\infty}h^j_t\,dB^j_t+n
\bigl(u^n_t-S_t\bigr)^-\,dt
\]
and
\[
dy^m_t+Ay^m_t\,dt=
\bar{f}_t\,dt+\sum_{i=1}^d
\partial_i\bar {g}^i_t\,dt+\sum
_{j=1}^{+\infty}\bar{h}^j_t\,dB^j_t+m
\bigl(y^m_t-\bar{S}_t\bigr)^-\,dt.
\]
From
the proofs above, we know that the penalized solution converges
weakly to the solution and we can take convex
combinations $\hat{u}^n=\sum_{i=1}^{N_n}\alpha_i^nu^{n_i}$
and $\hat{y}^n=\sum_{i=1}^{N'_n}\beta_i^n y^{n'_i}$ such that $(\hat
{u}^n )_n$ and
$(\hat{y}^n )_n$ are nondecreasing and converge strongly to $u$ and
$y$,
respectively, in $L^2 (\Omega\times[0,T], H^1_0 (\cO))$ as $n$ tends to
$+\infty$.

As in the proof of Theorem \ref{Itoformula}, we first establish a key lemma:
%
%le6 #&#
\begin{lemma}\label{crucial2} For all $t\in[0,T]$,
\[
\lim_{n\rightarrow+\infty}E\int_0^t\int
_\mathcal{O}\hat{u}^n_s \sum
_{k=1}^{N'_n}\beta^n_k
n'_k \bigl(y^{n'_k}_s-
\bar{S}_s\bigr)^-\,dx\,ds =E\int_0^t
\int_{\cO}\tilde{u}\bar{\nu} (ds,dx)
\]
and
\[
\lim_{n\rightarrow+\infty}E\int_0^t\int
_\mathcal{O}\hat{y}^n_s \sum
_{k=1}^{N_n}\alpha^n_k
n_k \bigl(u^{n_k}_s-S_s\bigr)^-\,
dx\,ds =E\int_0^t \int_{\cO
}
\tilde{y} {\nu} (ds,dx).
\]
\end{lemma}
\begin{pf} We put for all $n$,
\begin{eqnarray*}
\nu^n (ds,dx)&=& \sum_{k=1}^{N_n}
\alpha^n_k n_k \bigl(u^{n_k}_s-S_s
\bigr)^- \,dx\,ds\qquad\mbox{and}\\
\bar{\nu}^n (ds,dx)& =&\sum
_{k=1}^{N'_n}\beta^n_k
n'_k \bigl(y^{n'_k}_s-
\bar{S}_s\bigr)^-\,dx\,ds.
\end{eqnarray*}
As in the proof of Lemma \ref{crucial}, we write for all $n\in\bbN^*$:
$u^n =z+v^n$.

In the same spirit, we introduce $\bar{z}$ the solution of the linear SPDE:
\[
d\bar{z}_t +A\bar{z}_t =\bar{f}_t\,dt+\sum
_{i=1}^d\partial_i\bar
{g}^i_t\,dt+\sum_{j=1}^{+\infty}
\bar{h}^j_t\,dB^j_t,
\]
with initial condition $\bar{z}_0=\xi$ and put
$\forall n\in\mathbb{N}^*, \bar{v}^n =y^n -\bar{z}$, $\hat{\bar{v}}{}^n
=\hat{y}^n -\bar{z}$ and $\bar{v}=y-\bar{z}.$

As a consequence of Lemma II.6 in \cite{PIERRER}, we have for all
$n\in
\bbN^*$, $P$-almost surely,
\[
\frac{1}{2}\bigl\Vert \hat{v}^n_t-\hat{
\bar{v}}{}^n_t\bigr\Vert^2+\int_0^t
\mathcal{E}\bigl(\hat{v}^n_s -\hat{\bar{v}}{}^n_s
\bigr)\,ds=\int_0^t\int_\mathcal
{O}\bigl(\hat{v}^n_s - \hat{\bar{v}}{}^n_s
\bigr) \bigl({\nu}^n -{\bar{\nu}}^n \bigr) (dx,ds)
\]
and
\[
\frac{1}{2}\Vert {v}_t-{\bar{v}}_t
\Vert^2+\int_0^t\mathcal
{E}({v}_s -{\bar{v}}_s )\,ds=\int_0^t
\int_\mathcal{O}(\tilde{v}_s - \tilde {
\bar{v}}_s) ({\nu} -{\bar{\nu}} ) (dx,ds).
\]
But, as in the proof of Lemma \ref{crucial}, we get that $(\hat
{v}^n_t-\hat{\bar{v}}{}^n_t)_n$ tends to ${v}_t-{\bar{v}}_t$ in $L^2
(\cO
)$ almost surely and
\begin{eqnarray*}
\lim_n \int_0^t\int
_\mathcal{O}\hat{v}^n_s
\nu^n (dx,ds)&=& \int_0^t\int
_\mathcal{O}\tilde{v}_s \nu(dx,ds),
\\
\lim_n \int_0^t\int
_\mathcal{O}\hat{\bar{v}}{}^n_s \bar{
\nu}^n (dx,ds)&=& \int_0^t\int
_\mathcal{O}\tilde{\bar{v}}_s \bar{\nu}(dx,ds).
\end{eqnarray*}
This yields
\begin{eqnarray*}
&&\lim_n \biggl( \int_0^t\int
_\mathcal{O}\hat{v}^n_s \bar{
\nu}^n (dx,ds) +\int_0^t\int
_\mathcal{O}\hat{\bar{v}}{}^n_s
\nu^n (dx,ds) \biggr)\\
&&\qquad =\int_0^t\int
_\mathcal{O}\tilde{v}_s \bar{\nu}(dx,ds)+\int
_0^t\int_\mathcal {O}\tilde{
\bar{v}}_s \nu(dx,ds).
\end{eqnarray*}
But, we have
\begin{eqnarray*}
\limsup_n \int_0^t\int
_\mathcal{O}\hat{v}^n_s \bar{
\nu}^n (dx,ds)&\leq& \limsup_n \int_0^t
\int_\mathcal{O}v_s \bar{\nu}^n
(dx,ds)\\
&=&\int_0^t\int_\mathcal{O}
\tilde{v}_s \bar{\nu} (dx,ds),
\end{eqnarray*}
and in the same way
\[
\limsup_n \int_0^t\int
_\mathcal{O}\hat{\bar{v}}{}^n_s
\nu^n (dx,ds)\leq \int_0^t\int
_\mathcal{O}\tilde{\bar{v}}_s \nu(dx,ds).
\]
Let us remark that these inequalities also hold for any subsequence.
From this, it is easy to deduce that necessarily
\[
\lim_n \int_0^t\int
_\mathcal{O}\hat{v}^n_s \bar{
\nu}^n (dx,ds)=\int_0^t\int
_\mathcal{O}\tilde{v}_s \bar{\nu} (dx,ds)
\]
and
\[
\lim_n \int_0^t\int
_\mathcal{O}\hat{\bar{v}}{}^n_s
\nu^n (dx,ds)= \int_0^t\int
_\mathcal{O}\tilde{\bar{v}}_s \nu(dx,ds).
\]
We end the proof of this lemma by using similar arguments as in the
proof of Lemma \ref{crucial}.
\end{pf}
We now end the proof of Theorem \ref{Itodifference}.
We begin
with the equation which $\hat{u}^n-\hat{y}^n$ satisfies
\begin{eqnarray*}
&&d\bigl(\hat{u}^n_t-\hat{y}_t^n
\bigr)+A\bigl(\hat{u}^n_t-\hat {y}_t^n
\bigr)\,dt\\
&&\qquad=(f_t-\bar{f}_t)\,dt+\sum
_{i=1}^d\partial_i\bigl(g^i_t-
\bar {g}^i_t\bigr)\,dt+\sum_{j=1}^{+\infty}
\bigl(h^j_t-\bar{h}^j_t
\bigr)\,dB^j_t
+\bigl(\nu^n -\bar{\nu}^n \bigr) (x,dt ).
\end{eqnarray*}
Applying It\^o's formula to $\Phi(\hat{u}^n-\hat{y}^n)$, we have
\begin{eqnarray*}
&&\int_\mathcal{O}\Phi\bigl(t,\hat{u}_t^n(x)-
\hat {y}_t^n(x)\bigr)\,dx+\int_0^t
\mathcal{E}\bigl(\Phi'\bigl(s,\hat{u}_s^n-
\hat {y}_s^n\bigr),\hat {u}_s^n-
\hat{y}_s^n\bigr)\,ds\\
&&\qquad=\int_0^t
\bigl(\Phi'\bigl(s,\hat{u}_s^n-\hat
{y}_s^n\bigr),f_s-\bar{f}_s
\bigr)\,ds
\\
&&\qquad\quad{}-\sum_{i=1}^d\int_0^t
\int_\mathcal {O}\Phi''\bigl(s,
\hat{u}_s^n-\hat{y}_s^n\bigr)
\partial_i\bigl(\hat{u}_s^n-
\hat{y}_s^n\bigr) \bigl(g^i_s-
\bar{g}^i_s\bigr)\,dx\,ds\\
&&\qquad\quad{}+\sum_{j=1}^{+\infty}
\int_0^t\bigl(\Phi'\bigl(s,
\hat{u}_s^n-\hat {y}_s^n
\bigr),h^j_s-\bar {h}^j_s
\bigr)\,dB^j_s
\\
&&\qquad\quad{}+\frac{1}{2}\sum_{j=1}^{+\infty}\int
_0^t\int_\mathcal {O}
\Phi''\bigl(s,\hat{u}_s^n-
\hat{y}_s^n\bigr) \bigl(h_s^j-
\bar{h}^j_s\bigr)^2\,dx\,ds\\
&&\qquad\quad{}+\int
_0^t\int_\mathcal{O}
\frac{\partial
\Phi
}{\partial
s}\bigl(s,\hat{u}_s^n-
\hat{y}_s^n\bigr)\,dx\,ds
\\
&&\qquad\quad{}+\int_0^t\int_\mathcal{O}
\Phi'\bigl(s,\hat{u}_s^n-
\hat{y}_s^n\bigr) \bigl(\nu^n -\bar{
\nu}^n \bigr) (dx,dt )\qquad \mbox{a.s.}
\end{eqnarray*}
Because $(\hat{u}^n)_n$ and
$(\hat{y}^n)_n$ converge strongly to $u$ and $y$, respectively, the
convergence of all the terms except the last term are clear. For the
convergence of the last term, we do as follows:
\begin{eqnarray*}
&&\biggl\llvert \int_0^t\int
_\mathcal{O}\bigl[\Phi'\bigl(s,\hat
{u}_s^n-\hat{y}_s^n\bigr)-
\Phi'\bigl(s,u_s-\hat{y}_s^n
\bigr)\bigr]\nu^n(dx\,ds)\\
&&\quad{}+\int_0^t
\int_\mathcal{O}\bigl[\Phi'\bigl(s,u_s-
\hat{y}_s^n\bigr)-\Phi'(s,u_s-{y}_s)
\bigr]\nu^n(dx\,ds)\biggr\rrvert
\\
&&\qquad\leq C\int_0^t\int_\cO\bigl|
\hat{u}_s^n-u_s\bigr|\nu^n(dx\,ds)+
\int_0^t\int_\cO\bigl |\hat
{y}_s^n -y_s\bigr | \nu^n (dx,ds).
\end{eqnarray*}
As a consequence of Lemma \ref{crucial} and using the fact that $\hat
{u}^n \leq u$,
\[
\lim_n \int_0^t\int
_\cO\bigl|\hat{u}_s^n-u_s\bigr|
\nu^n(dx\,ds)=\lim_n \int_0^t
\int_\cO\bigl( u_s - \hat{u}_s^n
\bigr)\nu^n(dx\,ds)=0.
\]
By Lemma \ref{crucial2} and the fact that $\hat{y}^n \leq y$,
\[
\lim_n \int_0^t\int
_\cO\bigl|\hat{y}_s^n -y_s
\bigr| \nu^n (dx,ds)=\lim_n \int_0^t
\int_\cO\bigl(y_s -\hat{y}_s^n
\bigr) \nu^n (dx,ds)=0.
\]
This yields
\[
\lim_n \int_0^t\int
_\mathcal{O}\bigl(\Phi'\bigl(s,\hat{u}_s^n-
\hat {y}_s^n\bigr)- \Phi'
(s,u_s -y_s )\bigr)\nu^n (dx,dt )=0,
\]
but, by Lemma \ref{convergemeas}, we know that
\[
\lim_n \int_0^t\int
_\mathcal{O} \Phi' (s,u_s
-y_s )\nu^n (dx,dt )=\int_0^t
\int_\mathcal{O}\Phi'(s,\tilde{u}_s-
\tilde{y}_s)\bar{\nu }(dx,dt ),
\]
so
\[
\lim_n \int_0^t\int
_\mathcal{O}\Phi'\bigl(s,\hat{u}_s^n-
\hat {y}_s^n\bigr)\nu^n (dx,dt )= \int
_0^t\int_\mathcal{O}
\Phi'(s,\tilde{u}_s-\tilde {y}_s)\nu
(dx,dt ).
\]
In the same way, we prove
\[
\lim_n \int_0^t\int
_\mathcal{O}\Phi'\bigl(s,\hat{u}_s^n-
\hat {y}_s^n\bigr)\bar{\nu }^n (dx,dt )=
\int_0^t\int_\mathcal{O}
\Phi'(s,\tilde{u}_s-\tilde {y}_s)\bar{
\nu}(dx,dt ).
\]
The proof is now complete.
\end{pf}

%s5.5 #&#
\subsection{\texorpdfstring{Proof of Theorem \protect\ref{maintheo} in the nonlinear case}
{Proof of Theorem 4 in the nonlinear case}}
Let $\gamma$ and
$\delta$ be 2 positive constants. On
$L^2(\Omega\times[0,T];H^1_0(\mathcal{O}))$, we introduce the
norm
\[
\Vert u\Vert_{\gamma,\delta}=E\biggl(\int_0^Te^{-\gamma
s}
\bigl(\delta\Vert u_s\Vert^2+\Vert \nabla
u_s\Vert^2\bigr)\,ds\biggr),
\]
which clearly defines an equivalent norm on $L^2 (\Omega\times[0,T];
H^1_0 (\cO))$.

Let us consider the Picard sequence $(u^n)_n$
defined by $u^0=\xi$ and for all $n\in\mathbb{N}^*$ we denote by
$(u^{n+1},\nu^{n+1})$ the solution of the linear SPDE with obstacle
\[
\bigl(u^{n+1}, \nu^{n+1}\bigr)=\mathcal{R} \bigl(\xi, f
\bigl(u^n,\nabla u^n\bigr), g\bigl(u^n,
\nabla u^n\bigr), h\bigl(u^n,\nabla u^n
\bigr), S\bigr).
\]
Then, by It\^o's formula (\ref{itodifference}), we have almost
surely
\begin{eqnarray*}
&&e^{-\gamma T}\bigl\Vert u_T^{n+1}-u_T^n
\bigr\Vert^2+2\int_0^Te^{-\gamma
s}
\mathcal{E}\bigl(u_s^{n+1}-u_s^n
\bigr)\,ds\\
&&\qquad =-\gamma\int_0^Te^{-\gamma
s}\bigl\Vert
u_s^{n+1}-u_s^n
\bigr\Vert^2\,ds
\\
&&\qquad\quad{}+2\int_0^Te^{-\gamma
s}\bigl(
\hat{f}_s,u_s^{n+1}-u_s^n
\bigr)\,ds-2\sum_{i=1}^d\int
_0^Te^{-\gamma
s}\bigl(\hat{g}^i_s,
\partial_i\bigl(u_s^{n+1}-u_s^n
\bigr)\bigr)\,ds\\
&&\qquad\quad{}+2\sum_{j=1}^{+\infty
}\int
_0^Te^{-\gamma
s}\bigl(\hat{h}^j_s,u_s^{n+1}-u_s^n
\bigr)\,dB^j_s
\\
&&\qquad\quad{}+\int_0^Te^{-\gamma
s}\bigl\Vert |
\hat{h}_s|\bigr\Vert^2\,ds+2\int_0^T
\int_\mathcal {O}e^{-\gamma
s}\bigl(u_s^{n+1}-u_s^n
\bigr) \bigl(\nu^{n+1}-\nu^n\bigr) (dx\,ds),
\end{eqnarray*}
where
$\hat{f}=f(u^n,\nabla u^n)-f(u^{n-1},\nabla u^{n-1})$,
$\hat{g}=g(u^n,\nabla u^n)-g(u^{n-1},\nabla u^{n-1})$ and
$\hat{h}=h(u^n,\nabla u^n)-h(u^{n-1},\nabla u^{n-1})$. Clearly, the
last term is nonpositive, so using Cauchy--Schwarz's
inequality and the Lipschitz conditions on $f$, $g$ and $h$, we have
\begin{eqnarray*}
&&2\int_0^Te^{-\gamma s}\bigl(u_s^{n+1}-u_s^n,
\hat {f}_s\bigr)\,ds\\[-2pt]
&&\qquad\leq\frac{1}{\varepsilon}\int_0^Te^{-\gamma s}
\bigl\Vert u_s^{n+1}-u_s^n
\bigr\Vert^2\,ds+\varepsilon\int_0^T\Vert
\hat {f}_s\Vert^2\,ds
\\[-2pt]
&&\qquad\leq\frac{1}{\varepsilon}\int_0^Te^{-\gamma s}
\bigl\Vert u_s^{n+1}-u_s^n
\bigr\Vert^2\,ds+C\varepsilon\int_0^Te^{-\gamma s}
\bigl\Vert u_s^n-u_s^{n-1}
\bigr\Vert^2\,ds
\\[-2pt]
&&\qquad\quad{}+C\varepsilon\int_0^Te^{-\gamma s} \bigl\Vert
\nabla\bigl(u_s^n-u_s^{n-1}\bigr)
\bigr\Vert^2\,ds
\end{eqnarray*}
and
\begin{eqnarray*}
&&2\sum_{i=1}^d\int_0^Te^{-\gamma
s}
\bigl(\hat{g}^i_s,\partial_i
\bigl(u_s^{n+1}-u_s^n\bigr)
\bigr)\,ds\\[-2pt]
&&\qquad\leq2\int_0^Te^{-\gamma
s}\bigl\Vert
\nabla\bigl(u_s^{n+1}-u_s^n\bigr)\bigr\Vert
\bigl(C\bigl\Vert u_s^n-u_s^{n-1}\bigr\Vert
+\alpha\bigl\Vert \nabla \bigl(u_s^n-u_s^{n-1}
\bigr)\bigr\Vert \bigr)\,ds\\[-2pt]
&&\qquad\leq C\varepsilon\int_0^Te^{-\gamma
s}
\bigl\Vert \nabla\bigl(u_s^{n+1}-u_s^n
\bigr)\bigr\Vert^2\,ds+\frac{C}{\varepsilon
}\int_0^Te^{-\gamma
s}
\bigl\Vert u_s^n-u_s^{n-1}
\bigr\Vert^2\,ds
\\[-2pt]
&&\qquad\quad{}+\alpha\int_0^Te^{-\gamma
s}\bigl\Vert \nabla
\bigl(u_s^{n+1}-u_s^n\bigr)
\bigr\Vert^2\,ds+\alpha\int_0^Te^{-\gamma
s}
\bigl\Vert u_s^n-u_s^{n-1}
\bigr\Vert^2\,ds
\end{eqnarray*}
and
\begin{eqnarray*}
&&\int_0^Te^{-\gamma s}\bigl\Vert |
\hat{h}_s|\bigr\Vert^2\,ds\\[-2pt]
&&\qquad\leq C\biggl(1+\frac{1}{\varepsilon}
\biggr)\int_0^Te^{-\gamma s}\bigl\Vert
u_s^n-u_s^{n-1}
\bigr\Vert^2\,ds\\[-2pt]
&&\qquad\quad{} +\beta^2(1+\varepsilon)\int_0^Te^{-\gamma
s}
\bigl\Vert \nabla\bigl(u_s^n-u_s^{n-1}
\bigr)\bigr\Vert^2\,ds,
\end{eqnarray*}
where
$C$, $\alpha$ and $\beta$ are the constants in the Lipschitz
conditions. Using the elliptic condition and taking expectation, we
get
\begin{eqnarray*}
&&\biggl(\gamma-\frac{1}{\varepsilon}\biggr)E\int_0^Te^{-\gamma
s}
\bigl\Vert u_s^{n+1}-u_s^n
\bigr\Vert^2\,ds+(2\lambda-\alpha)E\int_0^Te^{-\gamma
s}
\bigl\Vert \nabla\bigl(u_s^{n+1}-u_s^n
\bigr)\bigr\Vert^2\,ds
\\[-2pt]
&&\qquad\leq C\biggl(1+\varepsilon+\frac{2}{\varepsilon}\biggr)\int_0^Te^{-\gamma s}
\bigl\Vert u_s^n-u_s^{n-1}
\bigr\Vert^2\,ds\\[-2pt]
&&\qquad\quad{}+\bigl(C\varepsilon+\alpha+\beta^2(1+\varepsilon
)\bigr)E\int_0^Te^{-\gamma
s}\bigl\Vert \nabla
\bigl(u_s^n-u_s^{n-1}\bigr)
\bigr\Vert^2\,ds.
\end{eqnarray*}
We
choose $\varepsilon$ small enough and then $\gamma$ such that
\[
C\varepsilon+\alpha+\beta^2(1+\varepsilon)<2\lambda-\alpha \quad\mbox{and}\quad
\frac
{\gamma
-1/\varepsilon}{2\lambda-\alpha}=\frac{C(1+\varepsilon+2/\varepsilon
)}{C\varepsilon
+\alpha+\beta^2(1+\varepsilon)}.
\]
If we set $\delta=\frac{\gamma-1/\varepsilon}{2\lambda-\alpha}$, we
have the following inequality:
\begin{eqnarray*}
\bigl\Vert u^{n+1}-u^n\bigr\Vert_{\gamma,\delta}&\leq&
\frac{C\varepsilon+\alpha
+\beta^2(1+\varepsilon)}{2\lambda-\alpha}\bigl\Vert u^n-u^{n-1}\bigr\Vert_{\gamma,\delta}\leq
\cdots \\
&\leq&\biggl(\frac{C\varepsilon
+\alpha
+\beta^2(1+\varepsilon)}{2\lambda-\alpha}\biggr)^n\bigl\Vert u^1
\bigr\Vert_{\gamma,\delta}
\end{eqnarray*}
when $n\rightarrow\infty$,
$(\frac{C\varepsilon+\alpha+\beta^2(1+\varepsilon)}{2\lambda-\alpha
})^n\rightarrow0$,
and we deduce that $(u^n)_n$ converges strongly to $u$ in
$L^2(\Omega\times[0,T];H_0^1(\mathcal{O}))$.

Moreover, as
$(u^{n+1},\nu^{n+1})=\mathcal{R} (\xi, f(u^n,\nabla u^n),
g(u^n,\nabla
u^n), h(u^n,\nabla u^n), S)$, we have for any $\varphi\in\mathcal{D}$,
\begin{eqnarray*}
\label{weaksol}&&\bigl(u^{n+1}_t,\varphi_t
\bigr)-(\xi,\varphi_0)-\int_0^t
\bigl(u^n_s,\partial_s\varphi_s
\bigr)\,ds+\int_0^t\mathcal {E}
\bigl(u^{n+1}_s,\varphi_s\bigr)\,ds\\
&&\quad{}+\sum
_{i=1}^d\int_0^t
\bigl(g^i_s\bigl(u_s^n,\nabla
u_s^n\bigr),\partial_i\varphi_s
\bigr)\,ds
\\
&&\qquad=\int_0^t\bigl(f_s
\bigl(u_s^n,\nabla u_s^n
\bigr),\varphi_s\bigr)\,ds+\sum_{j=1}^{+\infty
}
\int_0^t\bigl(h^j_s
\bigl(u_s^n,\nabla u_s^n
\bigr),\varphi_s\bigr)\,dB^j_s\\
&&\qquad\quad{}+\int
_0^t\int_{\mathcal{O}}
\varphi_s(x)\nu^{n+1}(dx\,ds)\quad \mbox{a.s.}
\end{eqnarray*}
Let $v^{n+1}$ be the random parabolic potential associated to $\nu^{n+1} $:
\[
\nu^{n+1}=\partial_t v^{n+1}+Av^{n+1}.
\]
We denote
$z^{n+1}=u^{n+1}-v^{n+1}$, so
\[
z^{n+1} =\mathcal{U}\bigl(\xi,f\bigl(u^n,\nabla
u^n\bigr), g\bigl(u^n,\nabla u^n\bigr), h
\bigl(u^n,\nabla u^n\bigr)\bigr)
\]
converges strongly to $z$
in $L^2(\Omega\times[0,T];H_0^1(\mathcal{O}))$. As a consequence of
the strong
convergence of $(u^{n+1})_n$, we deduce that $(v^{n+1})_n$ converges
strongly to $v$ in $L^2(\Omega\times[0,T];H_0^1(\mathcal{O}))$.
Therefore, for fixed $\omega$,
\begin{eqnarray*}
&&\int_0^t\biggl(-\frac{\partial_s\varphi_s}{\partial s},v_s
\biggr)\,ds+\int_0^t\mathcal {E}(
\varphi_s,v_s)\,ds\\
&&\qquad=\lim\int_0^t
\biggl(-\frac{\partial_s\varphi_s}{\partial
s},v^{n+1}_s\biggr)\,ds+\int
_0^t\mathcal{E}\bigl(\varphi_s,v^{n+1}_s
\bigr)\,ds\geq0,
\end{eqnarray*}
that is, $v(\omega)\in\mathcal{P}$. Then from Proposition \ref
{presentation},
we obtain a regular measure associated with $v$, and $(\nu^{n+1})_n$
converges vaguely to $\nu$.

Taking the limit, we obtain
\begin{eqnarray*}
&&(u_t,\varphi_t)-(\xi,\varphi_0)-\int
_0^t(u_s,\partial_s
\varphi_s)\,ds+\int_0^t
\mathcal{E}(u_s,\varphi_s)\,ds\\
&&\quad{}+\sum
_{i=1}^d\int_0^t
\bigl(g^i_s(u_s,\nabla u_s),
\partial_i\varphi_s\bigr)\,ds
\\
&&\qquad=\int_0^t\bigl(f_s(u_s,
\nabla u_s),\varphi_s\bigr)\,ds+\sum
_{j=1}^{+\infty
}\int_0^t
\bigl(h^j_s(u_s,\nabla u_s),
\varphi_s\bigr)\,dB^j_s\\
&&\qquad\quad{}+\int
_0^t\int_{\mathcal
{O}}
\varphi_s(x)\nu(dx,ds)\qquad \mbox{a.s.}
\end{eqnarray*}
From the fact that $u$ and $z$ are in $\mathcal{H}_T$, we know that
$v$ is also in $\mathcal{H}_T$, and by definition, $\nu$ is a random
regular measure. $ \hfill\Box$
%s6 #&#
\section{Comparison theorem}
%s6.1 #&#
\subsection{A comparison theorem in the linear case}
We first establish a comparison theorem for the solutions of
linear SPDE with obstacle in the case where the obstacles are the same;
this also
gives a comparison between the regular measures.

So, for this part only, we consider the same hypotheses as in
Section \ref{subsec52}. So we consider adapted processes $f$, $g$,
$h$, respectively, in $L^2 ([0,T]\times\Omega\times\cO;\mathbb{R})$, $L^2
([0,T]\times\Omega\times\cO;\mathbb{R}^d)$ and $L^2 ([0,T]\times
\Omega
\times\cO;\mathbb{R}^{\mathbb{N}^*})$, an obstacle $S$ which satisfies
assumption \ref{asso} and $\xi\in L^2(\Omega\times\mathcal{O})$ is an
$\mathcal{F}_0$-measurable random
variable such that $\xi\leq S_0$. We denote by $(u,\nu)$ the solution
of $\mathcal{R}(\xi, f, g, h, S)$.

We are given another $\xi'\in L^2(\Omega\times\mathcal{O})$ is
$\mathcal{F}_0$-measurable and such that $\xi'\leq S_0$ and another
adapted process $f'$ in $L^2 ([0,T]\times\Omega\times\cO;\mathbb{R})$.
We denote by $(u',\nu')$ the solution of $\mathcal{R}(\xi', f', g, h, S)$. We have the following comparison theorem:
%
%th7 #&#
\begin{theorem}
Assume that the following conditions hold:
\begin{longlist}[(1)]
\item[(1)]$\xi\leq\xi', dx\otimes dP$-a.e.
\item[(2)]$f\leq f', dt\otimes dx\otimes
dP$-a.e.
\end{longlist}
Then for almost all $\omega\in\Omega$, $u\leq u', q.e.$
and $\nu\geq\nu'$ in the sense of distribution.
\end{theorem}
\begin{pf} We consider the following two penalized
equations:
\begin{eqnarray*}
du_t^n&=&Au_t^n\,dt+f_t\,dt+
\sum_{i=1}^d\partial_ig^i_t\,dt+
\sum_{j=1}^{+\infty
}h^j_t\,dB^j_t+n
\bigl(u_t^n-S_t\bigr)^-\,dt,
\\
du_t^{\prime n}&=&Au_t^{\prime n}\,dt+f'_t\,dt+
\sum_{i=1}^d\partial_ig^i_t\,dt+
\sum_{j=1}^{+\infty}h^j_t\,dB^j_t+n
\bigl(u_t^{\prime n}-S_t\bigr)^-\,dt,
\end{eqnarray*}
and we denote
\begin{eqnarray*}
F_t\bigl(x,u_t^n\bigr)&=&f_t(x)+n
\bigl(u_t^n-S_t\bigr)^-,
\\
F'_t\bigl(x,u_t^n
\bigr)&=&f'_t(x)+n\bigl(u_t^n-S_t
\bigr)^-.
\end{eqnarray*}
With
assumption (2) we have that $F_t(x,u_t^n)\leq F'_t(x,u_t^n), dt\otimes dx\otimes dP$-a.e. Therefore, from the comparison theorem
for SPDE (without obstacle, see \cite{Denis}), we know that $\forall
t\in[0,T]$, $u_t^n\leq u_t^{\prime n}, dx\otimes dP$-a.e., thus,
$n(u_t^n-S_t)^-\geq n(u_t^{\prime n}-S_t)^-$.

The results are an immediate consequence of the construction of $(u,\nu
)$ and $(u',\nu' )$ given in Section \ref{subsec52}.
\end{pf}
%
%s6.2 #&#
\subsection{A comparison theorem in the general case}

We now come back to the general setting and consider $(u^1,\nu^1
)=\mathcal{R} (\xi^1, f^1,g,h,S^1)$ the solution of the SPDE with
obstacle with null boundary condition:
%
%e24 #&#
%e25 #&#
\begin{eqnarray}
\cases{ %
\displaystyle du^1_t(x)=Lu^1_t(x)\,dt+f^1
\bigl(t,x,u^1_t(x),\nabla u^1_t(x)
\bigr)\,dt\vspace*{2pt}\cr
\qquad\hspace*{22pt}{}\displaystyle+\sum_{i=1}^d\partial_ig_i
\bigl(t,x,u^1_t(x),\nabla u^1_t(x)
\bigr)\,dt
\nonumber
\vspace*{2pt}\cr
\qquad\hspace*{22pt}{}\displaystyle+\sum_{j=1}^{+\infty
}h_j
\bigl(t,x,u^1_t(x),\nabla u^1_t(x)
\bigr)\,dB^j_t +\nu^1 (x, dt),
\vspace*{2pt}\cr
u^1\geq S^1, u^1_0=
\xi^1,}
\end{eqnarray}
where we assume $(\xi^1, f^1,g,h)$ satisfy hypotheses \textup{\ref{assH}}, \textup{\ref{assI}} and \textup{\ref{asso}}.

We consider another coefficient $f^2$ which satisfies the same
assumptions as $f^1$, another obstacle $S^2$ which satisfies \textup{\ref{asso}}
and another initial condition $\xi^2$ belonging to $L^2 (\Omega\times
\cO)$ and $\mathcal{F}_0$ adapted such that $\xi^2\geq S^2_0$. We
denote by $(u^2,\nu^2 )=\mathcal{R} (\xi^2, f^2,g,h,S^2)$.
% the solution of the SPDE with obstacle
%u'_t(x))\,dt+\sum_{i=1}^d\partial_ig_i(t,x,u'_t(x),\nabla
%u'_t(x))\,dt\nonumber\\&\ \ \ \ +\sum_{j=1}^{+
%u'_t(x))\,dB^j_t +\nu' (x, dt)\\
%&u'\geq S\,\ u'_0=\xi',\ \end{split}\right.\end{equation}
%
%th8 #&#
\begin{theorem} \label{comparison}Assume that the following conditions
hold:
\begin{longlist}[(1)]
\item[(1)]$\xi^1\leq\xi^2, dx\otimes dP$-a.e.
\item[(2)]$f^1(u^1,\nabla u^1)\leq f^2(u^1,\nabla u^1), dt\otimes
dx\otimes dP$-a.e.
\item[(3)]$S^1\leq S^2, dt\otimes dx\otimes dP$-a.e.
\end{longlist}
Then for almost all $\omega\in\Omega$, $u^1(t,x)\leq u^2(t,x), q.e.$
\end{theorem}
We put $\hat{u}=u^1-u^2$, $\hat{\xi}=\xi^1-\xi^2$, $\hat
{f}_t=f^1(t,u^1_t,\nabla
u^1_t)-f^2(t,u^2_t,\nabla u^2_t)$, $\hat{g}_t=g(t,u^1_t,\nabla
u^1_t)-g(t,u^2_t,\nabla u^2_t)$ and $\hat{h}_t=h(t,u^1_t,\nabla
u_t)-h(t,u^2_t,\nabla u^2_t)$. The main idea is to evaluate
$E\Vert \hat{u}_t^+\Vert^2$, thanks to It\^o's formula, and
then apply Gronwall's inequality. Therefore, we start by the following lemma:
%
%le7 #&#
\begin{lemma}
For all $t\in[0,T]$, we have
%
%e26 #&#
\begin{eqnarray}
\label{positivepart}&&E\bigl\Vert \hat{u}_t^+\bigr\Vert^2+2E\int
_0^t\mathcal{E}\bigl(\hat{u}_s^+
\bigr)\,ds\nonumber\\
&&\qquad=E\bigl\Vert \hat{\xi }^+\bigr\Vert^2+2E\int_0^t
\bigl(\hat{u}_s^+,\hat{f}_s\bigr)\,ds-2E\int
_0^t\bigl(\nabla\hat {u}_s^+,\hat
{g}_s\bigr)\,ds
\\
&&\qquad\quad{}+2E\int_0^t\int_{\mathcal{O}}\hat
{u}_s^+(x) \bigl(\nu-\nu'\bigr) (dx\,ds)+E\int
_0^t\bigl\Vert I_{\{\hat{u}_s>0\}}|\hat{h}_s|
\bigr\Vert^2\,ds.\nonumber
\end{eqnarray}
\end{lemma}
\begin{pf} We approximate the function $\psi\dvtx y\in
R\rightarrow(y^+)^2$ by a sequence $(\psi_n)_{n\in\bbN^*}$ of
regular functions:
let $\varphi$ be a $C^{\infty}$ increasing function such that
\[
\forall y\in\,]-\hspace*{-2pt}\infty,1]\qquad \varphi(y)=0 \quad\mbox{and}\quad \forall y\in[2,+\infty[ \qquad\varphi(y)=1.
\]
We set for all $n\in\bbN^*$,
\[
\forall y\in R \qquad\psi_n(y)=y^2\varphi(ny).
\]
It is easy to verify that $(\psi_n)_n$
converges uniformly to the function $\psi$ and that, moreover, we have
the estimates
\[
\forall y\in R^+, \forall n \qquad 0\leq\psi_n(y)\leq\psi(y),\qquad 0\leq
\psi'_n(y)\leq Cy,\qquad \bigl|\psi''_n(y)\bigr|
\leq C.
\]
Thanks to Theorem \ref{Itodifference}, for all
$n\in\bbN^*$ and $t\in[0,T]$, we have
%
%e27 #&#
\begin{eqnarray}
\label{appropositivepart}&&E\int_{\mathcal{O}}\psi_n(
\hat{u}_s)\,dx+E\int_0^t\mathcal{E}
\bigl(\psi'_n(\hat{u}_s),
\hat{u}_s\bigr)\,ds\nonumber\\
&&\qquad =E\int_{\mathcal{O}}\psi_n(
\hat{\xi})\,dx+E\int_0^t\bigl(
\psi'_n(\hat {u}_s),\hat{f}_s
\bigr)\,ds
-E\int_0^t\bigl(\nabla\psi'_n(
\hat{u}_s),\hat {g}_s\bigr)\,ds\\
&&\qquad\quad{}+E\int_0^t
\int_{\mathcal{O}}\psi'_n\bigl(
\hat{u}_s(x)\bigr)\hat{\nu }(dx\,ds)+\frac{1}{2}E\int
_0^t\int_{\mathcal{O}}
\psi''_n\bigl(\hat {u}_s(x)
\bigr)\hat {h}_s^2(x)\,dx\,ds.
\nonumber
\end{eqnarray}
Taking the limit, thanks to the dominated convergence theorem, we
obtain the convergences of all the terms except $E\int_0^t\int_{\mathcal
{O}}\psi'_n(\hat{u}_s(x))\hat{\nu}(dx\,ds)$.

From \eqref{appropositivepart}, we know that
\[
-E\int_0^t\int_{\mathcal{O}}
\psi'_n\bigl(\hat{u}_s(x)\bigr)\hat{\nu
}(dx\,ds)\leq C.
\]
Moreover, we have the following relation:
\begin{eqnarray*}
&&-E\int_0^t\int_{\mathcal{O}}
\psi'_n\bigl(\hat{u}_s(x)\bigr)\hat{\nu
}(dx\,ds)
\\
&&\qquad=-E\int_0^t\int_{\mathcal{O}}
\psi'_n\bigl(S^1_s(x)-u^2_s(x)
\bigr)\nu^1(dx\,ds) \\
&&\qquad\quad{}+E\int_0^t\int
_{\mathcal{O}}\psi'_n\bigl(u^1_s(x)-S^2_s(x)
\bigr)\nu^2(dx\,ds)
\\
&&\qquad=E\int_0^t\int_{\mathcal{O}}
\psi'_n\bigl(u^2_s(x)-S^1_s(x)
\bigr)\nu^1(dx\,ds)\\
&&\qquad\quad{}+ E\int_0^t\int
_{\mathcal{O}}\psi'_n\bigl(u^1_s(x)-S^2_s(x)
\bigr)\nu^2(dx\,ds).
\end{eqnarray*}
By Fatou's lemma, we obtain
\begin{eqnarray*}
&&2E\int_0^t\int_{\mathcal{O}}
\bigl(u^2_s(x)-S^1_s(x)\bigr)^+
\nu^1(dx\,ds)+2E\int_0^t\int
_{\mathcal{O}}\bigl(u^1_s(x)-S^2_s(x)
\bigr)^+\nu^2(dx\,ds)\\
&&\qquad<+\infty.
\end{eqnarray*}
Hence, the convergence of the term $E\int_0^t\int_{\mathcal{O}}\psi'_n(\hat{u}_s(x))\hat{\nu}(dx\,ds)$ comes from the dominated
convergence theorem.
\end{pf}

\begin{pf*}{Proof of Theorem \ref{comparison}}
Applying It\^o's formula
(\ref{positivepart}) to $(\hat{u}_t^+)^2$, we have
\begin{eqnarray*}
&&E\bigl\Vert \hat{u}_t^+\bigr\Vert^2+2E\int_0^tI_{\{
\hat
{u}_s>0\}}
\mathcal{E}(\hat{u}_s)\,ds\\
&&\qquad=2E\int_0^t
\bigl(\hat{u}_s^+,\hat {f}_s\bigr)\,ds+2E\int
_0^t\bigl(\hat{u}_s^+,
\hat{g}_s\bigr)\,ds
\\
&&\qquad\quad{}+E\int_0^t\bigl\Vert I_{\{\hat{u}_s>0\}}|
\hat{h}_s|\bigr\Vert^2\,ds+2E\int_0^t
\int_{\mathcal
{O}}\bigl(u^1_s-u^2_s
\bigr)^+(x) \bigl(\nu^1-\nu^2\bigr) (dx,ds).
\end{eqnarray*}
As we assume that $f^1(u^1,\nabla u^1)\leq f^{2}(u^1,\nabla
u^1)$,
\begin{eqnarray*}
\hat{u}_s^+\hat{f}_s&=&\hat{u}_s^+\bigl\{
f^1\bigl(s,u^1_s,\nabla u^1_s
\bigr)-f^2\bigl(s,u^1_s,\nabla
u^1_s\bigr)\bigr\}\\
&&{}+\hat{u}_s^+\bigl
\{f^2\bigl(s,u^1_s,\nabla
u^1_s\bigr)-f^2\bigl(s,u^2_s,
\nabla u^2_s\bigr)\bigr\}
\\
&\leq&\hat{u}_s^+\bigl\{ f^2\bigl(s,u^1_s,
\nabla u^1_s\bigr)-f^2\bigl(s,u^2_s,
\nabla u^2_s\bigr)\bigr\}.
\end{eqnarray*}
Then with the
Lipschitz condition, using Cauchy--Schwarz's inequality, we have the
following relations:
\begin{eqnarray*}
E\int_0^t\bigl(\hat{u}_s^+,
\hat{f}_s\bigr)\,ds&\leq& \biggl(C+\frac{C}{\varepsilon}\biggr)E\int
_0^t\bigl\Vert \hat{u}_s^+
\bigr\Vert^2\,ds+\frac
{C\varepsilon}{\lambda} E\int_0^t
\mathcal{E}\bigl(\hat{u}_s^+\bigr)\,ds,
\\
E\int_0^t\bigl(\nabla\hat{u}_s^+,
\hat{g}_s\bigr)&\leq&\frac{\varepsilon+\alpha
}{\lambda
}E\int_0^t
\mathcal{E}\bigl(\hat{u}_s^+\bigr)\,ds+\frac{C}{\varepsilon}E\int
_0^t\bigl\Vert \hat{u}_s^+
\bigr\Vert^2\,ds,
\\
E\int_0^t\bigl\Vert I_{\{\hat{u}_s>0\}}|
\hat{h}_s|\bigr\Vert^2\,ds&\leq& CE\int_0^t
\bigl\Vert \hat{u}_s^+\bigr\Vert^2\,ds+\frac{\beta^2+\varepsilon
}{\lambda}E\int
_0^t\mathcal{E}\bigl(\hat{u}_s^+
\bigr)\,ds.
\end{eqnarray*}
The last term is equal to
$-2E\int_0^t\int_{\mathcal{O}}(u^1_s-u^2_s)^+(x)\nu^2(dx,ds)\leq0$,
because on $\{u^1\leq u^2\}$, $(u^1-u^2)^+=0$ and on $\{u^1>u^2\}$,
$\nu^1(dx,ds)=0$. Thus, we have the following inequality:
\[
E\bigl\Vert \hat{u}_t^+\bigr\Vert^2+\biggl(2-\frac{2\alpha+2\varepsilon
}{\lambda
}-
\frac{2C\varepsilon}{\lambda}-\frac{\beta^2+\varepsilon}{\lambda
}\biggr)E\int_0^t
\mathcal{E}\bigl(\hat{u}_s^+\bigr)\,ds\leq CE\int_0^t
\bigl\Vert \hat {u}_s^+\bigr\Vert^2\,ds.
\]
We can take $\varepsilon$ small enough such that
$2-\frac{2\alpha+2\varepsilon}{\lambda}-\frac{2C\varepsilon}{\lambda
}-\frac
{\beta^2+\varepsilon}{\lambda}>0$,
and we have
\[
E\bigl\Vert \hat{u}_t^+\bigr\Vert^2\leq CE\int
_0^t\bigl\Vert \hat{u}_s^+
\bigr\Vert^2\,ds.
\]
Then we deduce the
result from Gronwall's lemma.
\end{pf*}

%re6 #&#
\begin{remark}Applying the comparison theorem to the same obstacle
gives another proof of the uniqueness of the solution.
\end{remark}

%The authors wish to thank the anonymous referee for the pertinent
%remarks she/he made.
%

% imsref loaded by akundreckaite, 2013-02-14 08:28:53
%


%suskaldyti doi

\printaddresses

\end{document}